\documentclass[11pt, a4paper]{amsart}
\usepackage{enumerate}
\usepackage{amsthm}
\usepackage[dvipsnames]{xcolor}
\usepackage{graphicx}
\usepackage[margin=2.5cm]{geometry}
\usepackage{soul}
\usepackage[numeric]{amsrefs}
\usepackage{enumerate}

\usepackage[colorlinks=true, allcolors=blue]{hyperref}

\DeclareMathOperator{\aut}{Aut}
\DeclareMathOperator{\im}{Im}
\DeclareMathOperator{\re}{Re}
\DeclareMathOperator{\ord}{ord}
\DeclareMathOperator{\id}{I}
\DeclareMathOperator{\mcg}{Mod}
\DeclareMathOperator{\U}{U}
\DeclareMathOperator{\lift}{Lift}
\DeclareMathOperator{\Cap}{Cap}
\DeclareMathOperator{\conf}{Conf}

\newcommand{\zz}{\mathbb{Z}}
\newcommand{\pp}{\mathbb{CP}}
\newcommand{\cc}{\mathbb{C}}
\newcommand{\C}{\mathcal{C}}
\newcommand{\B}{\mathcal{B}}
\newcommand{\D}{\mathcal{D}}
\newcommand{\tD}{\wt\D}
\newcommand{\E}{\mathcal{E}}
\newcommand{\tE}{\wt\E}
\newcommand{\Ts}{\tilde{s}}
\newcommand{\Tt}{\tilde{t}}
\newcommand{\inner}[1]{\left\langle #1 \right \rangle}
\newcommand{\ov}{\overline}
\newcommand{\wt}{\widetilde}
\newcommand{\tx}[1]{\qquad \text{#1} \qquad}

\newcommand{\mat}[1]{\left(
\begin{smallmatrix}
  #1
\end{smallmatrix} \right)
}
\newcommand{\Mat}[1]{
\begin{pmatrix}
  #1
\end{pmatrix} 
}

\newtheoremstyle{customstyle}{3em}{2em}{\itshape}{}{\bfseries}{.}{.5em}{}

\theoremstyle{customstyle}
\newtheorem{theorem}{Theorem}[section]
\newtheorem{lemma}[theorem]{Lemma}
\newtheorem{proposition}[theorem]{Proposition}
\newtheorem{corollary}[theorem]{Corollary}

\theoremstyle{plain}
\newtheorem{question}[theorem]{Question}

\theoremstyle{definition}

\theoremstyle{remark}
\newtheorem{remark}[theorem]{Remark}

\counterwithin{equation}{section}
\title{Monodromy Representations of Mixed Braid Groups}
\author{Chitrabhanu Chaudhuri, Saswati Mukherjee }
\date{January 2026}

\begin{document}

\maketitle

\begin{abstract}
  We explicitly describe unitary representations of mixed braid groups
  on the cohomology of Abelian branched covers of $\pp^1$. We show that the 
  image of the representation is generated by complex reflections and relate 
  it to the multivariate Burau representation.
\end{abstract}

\setcounter{tocdepth}{1}
\tableofcontents

%==================First Section====================================================

\section{Introduction}

The action of the braid group on the cohomology of covering spaces of the punctured 
plane is a classical subject. Note that a tuple $p = (p_1, \ldots, p_n)$ in $\cc^n$ of $n$ 
distinct points determines a compact Riemann surface $X_p$ with affine part
\begin{equation*}
  \{(x,y) \in \cc^2 \mid y^d = (x-p_1)(x-p_2)\cdots (x-p_n) \}.
\end{equation*}
The map $X_p \to \pp^1$ extending $(x,y) \mapsto x$ is a Galois (branched) covering 
with cyclic Galois group $\zz/d\zz$. Now if we vary $p$ over the unordered 
configuration space
\begin{equation*}
  \conf_n(\cc) = \frac{\{(z_1, \ldots, z_n)\in \cc^n \mid z_i \neq z_j\}}{S_n}
\end{equation*}
we have a (smooth) fiber bundle $\pi: \mathfrak{X} \to \conf_n(\cc)$ with fibers $\pi^{-1}(p) = X_p$. 
Let $\cc_\mathfrak{X}$ be the (locally) constant sheaf on $\mathfrak{X}$.
Taking the first higher direct image sheaf:
\begin{equation*}
  \mathcal{V} = R^1 \pi_* \mathbb{C}_{\mathfrak{X}}
\end{equation*}
we obtain a local system $\mathcal{V}$ over $\conf_n(\cc)$ with fibers
\begin{equation*}
  \mathcal{V}_p = H^1(X_p, \cc).
\end{equation*}
Consequently we have a monodromy representation of the Braid group
\begin{equation*}
  \B_n = \pi_1(\conf_n(\cc)) \to \mathrm{GL}(H^1(X_p,\cc)).
\end{equation*}
In \cite{Mc}, McMullen gives a complete account of this representation, and 
the geometric structures on moduli space, that arise via this branched-cover 
construction.

This paper generalises this construction from cyclic covers to Abelian covers. Let 
$\Lambda = (n_1, \ldots, n_m)$ be a partition of $n$. Consider the compact Riemann 
surface $X_p$ which completes
\begin{equation*}
  \{ (x,y_1, \ldots, y_m) \in \cc^{m+1} \mid y_j^{d_j} = 
                   (x-p_{h_{j-1}+1}) \cdots (x - p_{h_j}) \}
  \tx{where} 
  h_j = \sum_{i=1}^j n_i.
\end{equation*}
Now the map $X_p \to \pp^1$, $(x,y_1,\dots, y_m) \mapsto x$ is a Galois cover with 
abelian Galois group isomorphic to $\zz/d_1\zz \times \cdots \times \zz/d_m\zz$.
The Galois group is generated by 
\begin{equation*}
  T_1, \ldots, T_m \tx{where}
  T_j(x,y_1, \ldots, y_m) = (x, y_1,\ldots, \zeta_{d_j}y_j, \ldots, y_m).
\end{equation*}
Consider the variation of the point $p = (p_1, \ldots, p_n)$ within the configuration 
space $\conf_{n,\Lambda}(\cc)$. This space is defined as the quotient:
\begin{equation*}
  \conf_{n,\Lambda}(\cc) = 
  \frac{\{(z_1, \ldots, z_n)\in \cc^n \mid z_i \neq z_j\}}{S_{n,\Lambda}}.
\end{equation*}
Here, $S_{n,\Lambda}$ is the subgroup of the symmetric group $S_n$ containing all 
permutations that preserve the partition $\Lambda$. This results in a 
fiber bundle $\pi: \mathfrak{X} \to \conf_{n,\Lambda}(\cc)$, where each fiber 
is given by $X_p$. In this way we get a monodromy representation
\begin{equation} \label{eq:monodromy}
  \theta: \pi_1(\conf_{n,\Lambda}(\cc)) \to \mathrm{GL}(H^1(X_p, \cc)).
\end{equation}
The fundamental group $\B_{n,\Lambda}:= \pi_1(\conf_{n,\Lambda}(\cc))$ is called the mixed 
braid group or coloured braid group for the partition $\Lambda$. 

Let us fix $p \in \conf_{n,\Lambda}(\cc)$ and write $X := X_p$. The representation $\theta$ 
preserves the hermitian intersection form 
\begin{equation*}
  \inner{ \xi_1, \xi_2 } = \int_{X} \xi_1 \wedge \ov \xi_2, \qquad
  \xi_1, \xi_2 \in H^1(X)   
\end{equation*} 
and commutes with the Galois group $\zz/d_1\zz \times \cdots \times \zz/d_m\zz$.
Let us decompose $H^1(X)$ into $T_1, \ldots, T_m$ eigenspaces
\begin{equation*}
  H^1(X) = \bigoplus_{\rho = (\rho_1, \ldots, \rho_m)} 
             \Big( H^1(X)_\rho := \cap_j \ker(T_j - \rho_j\id) \Big)
\end{equation*}
where the direct sum is over $\rho = (\rho_1, \ldots \rho_m) \in \cc^m$, 
with $\rho_j^{d_j} = 1$. The representation  $\theta$ preserves each eigenspace
$H^1(X)_\rho$. Restricting the intersection form to $H^1(X)_\rho$ 
we thus get a unitary representation
\begin{equation*}
  \theta_{\rho}: \B_{n,\Lambda} \to \U(H^1(X)_\rho) \cong \U(r,s)
\end{equation*}
where $(r,s)$ is the signature of the hermitian form on $H^1(X)_\rho$. 
Here $\U(r,s)$ is the pseudo-unitary group of signature $(r,s)$.

\subsection*{Description of results}
In \eqref{eq:signature} we compute the signature using the Chevalley-Weil formula
to find that when $\rho = (\zeta_{d_1}^{-k_1}, \ldots, \zeta_{d_j}^{-k_m})$
and $0 < k_j < d_j$ for each $j$,
\begin{align*}
  r & = \left\lceil \sum_{j=1}^m \frac{n_jk_j}{d_j} - 1 \right\rceil, &
  s & = \left\lceil n- \sum_{j=1}^m \frac{n_j k_j}{d_j} - 1 \right\rceil
\end{align*} 
and in this case, $\dim H^1(X)_\rho = n-1$ when 
$\rho_1^{n_1}\cdots \rho_m^{n_m} \neq 1$,
otherwise $n-2$ (here $\zeta_{d_j} = e^{2\pi \sqrt{-1}/d_j}$).

The crux of the paper is finding an explicit spanning set $\omega_1, \ldots, \omega_{n-1}$ for 
$H^1(X)_\rho$ in which the intersection form is given by 
\begin{align*} 
  \inner{\omega_i, \omega_i}     
    & =  \sqrt{-1} (\rho_j+1)(\rho_j-1)^{-1}
    & h_{j-1} < i < h_j, \\[.25em]
  %--------------------------------------------------------------------------
    \inner{\omega_{h_j}, \omega_{h_j}} 
      & = \sqrt{-1} (\rho_j\rho_{j+1} - 1)(\rho_j-1)^{-1}(\rho_{j+1}-1)^{-1}         
      & 1 \le j < m, \\[.25em]
  %--------------------------------------------------------------------------
    \inner{\omega_{i-1}, \omega_i} 
      & = \sqrt{-1}(-\rho_j)(1-\rho_j)^{-1} 
      & h_{j-1} < i \le  h_j.
\end{align*}
The spanning set is obtained by taking weighted sums of Poincar\'e duals of lifts of 
certain curves in $\cc - \{p_1, \ldots, p_n\}$ (see Figure \ref{fig:gencurv}
and definitions \eqref{eq:gens}).

\textbf{The central result of the paper is Theorem \ref{thm:main}} where we describe 
the representation relative to the spanning set above utilising the intersection form.
This description holds for any $\rho$ satisfying $\rho_j \neq1$ and 
$\rho_j\rho_k \neq 1$ for any $j,k \in \{ 1, \ldots, m \}$. 
These conditions account for the generic case for large $d_j$; indeed, if the $d_j$ 
are pairwise co-prime, the second condition holds by default. Furthermore, 
in Proposition \ref{prop:irred}, we demonstrate that \textbf{$\theta_\rho$ is irreducible 
under these assumptions}.

The representation $\theta_\rho$ is not faithful, in fact in Proposition 
\ref{prop:kernel}, we show that if $\tau = (\sigma_1 \cdots \sigma_{n-1})^n$,
that is the generator of the center of $\B_n$, (where $\sigma_i$ are the standard
generators of the braid group), then $\tau^f$ is always in the kernel of
$\theta_\rho$ where $f = \ord( n_1T_1 + \cdots + n_m T_m)$.

We show that the representation \textbf{$\theta_\rho$ is isomorphic to the 
dual of the reduced multivariate Burau representation
evaluated at $\ov\rho$} (see Proposition \ref{prop:burau}). When $\Lambda = (n)$ the 
representation $\theta_\rho$ is the reduced Burau representation for the braid group 
$\B_n$ evaluated at $\ov\rho$ (see Theorem 5.5 of \cite{Mc}). At the other extreme 
for $\Lambda = (1, \ldots, 1)$ we get the dual of reduced Gassner representation
evaluated at $\ov \rho$. Thus our family of representations interpolate between the 
reduced Burau and the reduced Gassner representations. 

Finally we observe that
$\theta_\rho(\B_{n,\Lambda})$ is infinite whenever for some $j$, $n_j \geq 3$, 
$\rho_j$ is a primitive $d_j$-th root of unity and
\begin{equation*}
  (n_j,d_j) \not\in \{ (3,3), (3,4), (3,6), (3,10), (4,4), (4,6), (5,6), (6,6)\}.
\end{equation*}

\subsection*{Connections with related literature} 
It has been shown by Murakami \cite{Mur}, Morton \cite{Mor}, and Gabrov\v{s}ek-Horvat 
\cite{GH} that the multivariate Burau representation can be used to calculate the 
multivariate Alexander polynomial of links. This establishes a direct connection 
between the representations investigated in this paper and knot theory, 
suggesting significant potential for future research into these connections.

The topological framework of this paper is deeply tied to the classical Deligne-Mostow 
theory of hyper-geometric monodromy \cite{DM}. In Section 12, Deligne-Mostow study
monodromy representation of cyclic covers of $\pp^1$ of the form
\begin{equation*}
  y^d  = (x-a_1)^{k_1} \cdots (x-a_n)^{k_n}
\end{equation*}
primarily to classify which rational weights $(k_1/d, \ldots, k_n/d)$ yield 
discrete arithmetic or non-arithmetic lattices in the complex hyperbolic isometry 
group $\mathrm{PU}(1, N)$. This is related to the image of $\theta_\rho$ when 
when $\rho = (\zeta_d^{k_1}, \ldots, \zeta_d^{k_m})$ and 
$\lceil n - \sum_i k_i/d - 1\rceil = 1$.

The images of the Burau and Gassner representations were studied  
extensively by Venkataramana in \cite{Ven14} and \cite{Ven14a}, where he demonstrated 
that these images are arithmetic in a wide range of cases.

In a related direction, Looijenga \cite{Loo} investigated Abelian covers of surfaces of 
genus at least 2. He demonstrated that the image of the representation of the liftable 
mapping class group on the cohomology of the cover is an arithmetic group in all but a 
few cases in genus 2.

\subsection*{Future directions} 
It will be very interesting to investigate the image and kernel of $\theta_\rho$ for 
different partitions $\Lambda$ and different roots of unity $\rho$. 
The following question is open but notoriously difficult:
\begin{question}
  What is the kernel of $\theta_\rho$? What is the image 
  $\theta_\rho(\B_{n,\Lambda})$ in $\U(r,s)$?
\end{question}
An exhaustive account seems out of reach right now but small cases 
can be analysed to start with. We plan to take this up in a subsequent paper.

In a different direction it will be worthwhile to explore the structures on the 
moduli space $M_{0,n}$ of genus zero Riemann surfaces with $n$ marked points that 
are induced by $\theta_\rho$. This was explored in the cyclic case by McMullen  in 
sections 7 and 8 of \cite{Mc}.

Finally in the future we would like to generalise our work to the most general 
Abelian covers 
\begin{equation*}
  \Big\{ (x,y_1, \ldots, y_m) \in \cc^{m+1} \mid y_j^{d_j} = 
                   (x-p_{h_{j-1}+1})^{a_{h_{j-1}+1}}  \cdots 
                   (x - p_{h_j})^{a_{h_j}} \Big\}.
\end{equation*}
 
%==================Second Section===================================================

\section{Topological Construction of the Monodromy}

\subsection*{Mixed braid groups}
Let $\B_n$ be Artin's braid group on $n$ strands defined as,
\begin{equation*}
    \B_n := \Pi_1((\mathbb{C}^n\setminus \Delta)/S_n)
\end{equation*}
where $\Delta=\{(z_1,...,z_n)\in \mathbb{C}^n \mid z_i = z_j \text{ for some } 
i \neq j\}$ is the big diagonal and $S_n$ is the symmetric group of degree $n$.

The braid group admits the standard presentation
\begin{align*}
  \B_n = \langle \sigma_1, \ldots, \sigma_{n-1} \mid 
       & \ \sigma_i \sigma_j = \sigma_j \sigma_i \quad  
         \forall \ 1 \leq i < j-1 < n-1, \\
       & \ \sigma_i \sigma_{i+1} \sigma_i = \sigma_{i+1}\sigma_i \sigma_{i+1} \quad
         1 \leq i < n-1  \rangle
\end{align*}
where the generators are the braids:
\begin{center}
  \includegraphics[height=3cm]{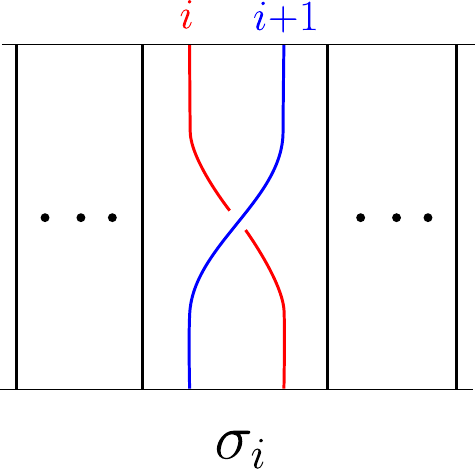}
\end{center}

There is a natural surjective homomorphism 
\begin{equation*} \label{eq:braidsym}
  \Phi_n: \B_n \to S_n, \qquad \Phi_n(\sigma_i) = (i,i+1)
\end{equation*}
which maps a braid to the permutation it induces on its end points.

Let $\Lambda$ be a partition of the set $\{1, \ldots, n \}$ and denote by 
$S_{n,\Lambda}$ the subgroup of $S_n$ consisting of all permutations 
that preserve $\Lambda$. The \textbf{mixed braid group} $\B_{n,\Lambda}$ corresponding 
to $\Lambda$  is defined as
\begin{equation*}
  \B_{n,\Lambda} := \Phi_n^{-1} (S_{n,\Lambda}).
\end{equation*}
It turns out that 
\begin{equation*}
    \B_{n,\Lambda} = \Pi_1((\mathbb{C}^n\setminus \Delta)/S_{n,\Lambda}).
\end{equation*} 
 
Since $\Lambda = \{\Lambda_1, \ldots, \Lambda_m \}$ is a partition of 
$\{1, \ldots,n\}$, we may, after 
relabelling the elements, assume that each block consists of consecutive integers. 
That is,
\begin{equation*}
  \Lambda_i = \{h_{i-1}+1, \ldots, h_i\},
\end{equation*}
for some integers $0=h_0 < h_1 < \ldots < h_m=n$. For convenience of notation
we shall denote the partition as a tuple $\Lambda = (n_1, \ldots, n_m)$ where 
$n_j = h_j-h_{j-1}$.

\subsection*{Generators and relations}
Manfredini \cite[Theorem 4]{Man} gives a presentation of $\B_{n,\Lambda}$. In that 
presentation the generators are
\begin{equation*}
  \sigma_i \quad 1 \le i <n, \ i \neq h_j, \ 1 \le j < m, \tx{and}
  A_{j,k} \quad 1 \le j < k \le m,
\end{equation*}
where $\sigma_i$ are the standard generators of the braid group and
\begin{equation*}
  A_{j,k} = (\sigma_{h_{k-1}} \sigma_{h_{k-1}-1}\cdots \sigma_{h_j+1})
            \, \sigma_{h_j}^2 \,
            (\sigma_{h_{k-1}} \sigma_{h_{k-1}-1}\cdots \sigma_{h_j+1})^{-1}.
\end{equation*}
\begin{center}
  \includegraphics[height=3cm]{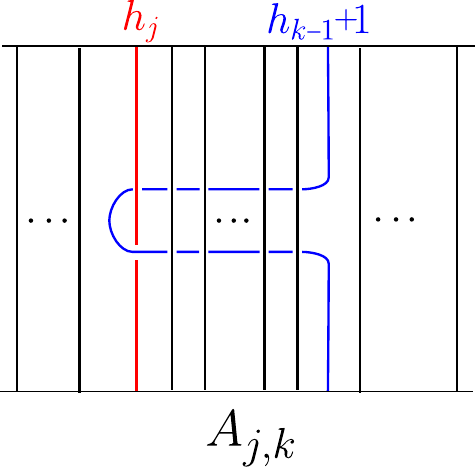}
\end{center}
Hence $A_{j,k}$ only moves the last element of $\Lambda_j$ and first element of 
$\Lambda_k$.
  
We do not list the relations in general, however, when $m=2$ the presentation 
simplifies significantly:\\[5pt]
\underline{\textsl{Generators} (m=2):} $\sigma_1, \ldots, \sigma_{h_1 -1}, \, 
A_{1,2} = \sigma_{h_1}^2,\, \sigma_{h_1+1}, \ldots, \sigma_{n-1}$\\[5pt]
\underline{\textit{Relations} (m=2):}
\begin{align*}
    \sigma_i \sigma_j & = \sigma_j \sigma_i & i<j-1,\\
    \sigma_i\sigma_{i+1}\sigma_i & = \sigma_{i+1}\sigma_i \sigma_{i+1}  \\
    \sigma_i A_{1,2}  & = A_{1,2}\sigma_i   & i \not\in \{h_1-1, h_1+1\}, \\
    \sigma_{h_1-1}A_{1,2}\sigma_{h_1-1}A_{1,2}  
       & = A_{1,2}\sigma_{h_1-1} A_{1,2}\sigma_{h_1-1} \\
    \sigma_{h_1+1}A_{1,2}\sigma_{h_1+1}A_{1,2}  
       & = A_{1,2}\sigma_{h_1+1} A_{1,2}\sigma_{h_1+1} \\
    (\sigma_{h_1-1}A_{1,2}\sigma_{h_1-1}^{-1})
    (\sigma_{h_1+1}A_{1,2}\sigma_{h_1+1}^{-1}) & =
    (\sigma_{h_1+1}A_{1,2}\sigma_{h_1+1}^{-1})
    (\sigma_{h_1-1}A_{1,2}\sigma_{h_1-1}^{-1}).
\end{align*}

\subsection*{Relation to Mapping class groups}
For distinct points $p_1, \ldots, p_n \in \cc$ there is an isomorphism 
\begin{equation} \label{eq:isom}
  \B_n \cong \mcg_c(\cc, \{p_1, \ldots, p_n\})
\end{equation}
where $\mcg_c(\cc, \{p_1, \ldots, p_n\})$ is the group of isotopy classes of compactly 
supported homeomorphisms of $\cc$ that fix the set $\{p_1, \ldots, p_n\}$. This
isomorphism takes the generator $\sigma_i$ to the half twist of a disk containing
$p_i, p_{i+1}$ as in Figure \ref{fig:generators} (see \cite[Section 1.6.2]{KT}).
\begin{figure}[h]
	\begin{center}
	  \begin{tabular}{ccc}
	    \includegraphics[height=4cm]{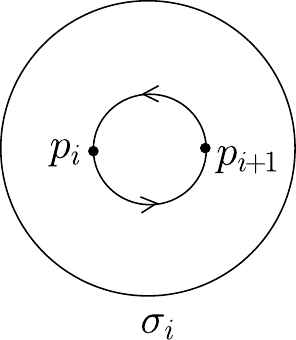} & \hspace{1cm} &
	    \includegraphics[height=4cm]{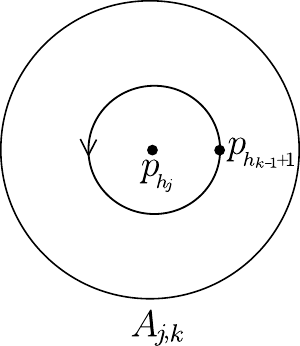}
	  \end{tabular}
	\end{center}
	\caption{Generators of $\B_{n,\Lambda}$ seen as mapping classes.}
	\label{fig:generators}
\end{figure}

Under this isomorphism the generators $A_{j,k}$ of $\B_{n,\Lambda}$ can be realised 
as the mapping class of the full twist in a disk containing 
$p_{h_j}, p_{h_{k-1}+1}$ with $p_{h_j}$ as the center. We shall use the 
identification $\eqref{eq:isom}$ throughout this paper.

\subsection*{Abelian Covers of $\pp^1$}
Let $\Gamma$ be the finite Abelian group
\begin{equation*} 
  \Gamma := \zz/d_1\zz \times \cdots \times \zz/d_m\zz
\end{equation*}
with standard generators $T_1, \ldots, T_m$.
Let $p_1, \ldots, p_n \in \cc$ be $n$ distinct points. Note that 
the fundamental group of $\cc$ punctured at $p_1, \ldots, p_n$ is the free group
\begin{equation*}
  \Pi_1(\cc -\{p_1, \ldots, p_n\}) = \langle c_1, \ldots, c_n \rangle
\end{equation*}
where $c_i$ is a small counter-clockwise loop encircling $p_i$. We consider 
a surjective homomorphism 
\begin{equation*}
  \Psi : \Pi_1(\cc -\{p_1, \ldots, p_n\}) \to \Gamma, \quad \text{where} \quad
  \Psi(c_i) = T_j \text{ for some } j.
\end{equation*}
There is a regular cover of $\cc -\{p_1, \ldots, p_n\}$ with Deck group $\Gamma$ 
associated to this homomorphism. 
By Riemann existence theorem, this cover extends to a branched Galois cover 
\begin{equation*}
  \pi: X \to \pp^1 \tx{with} \aut_{\pp^1}(X) \cong \Gamma.
\end{equation*}
This cover is ramified over $p_1, \ldots, p_n$ and infinity.

After relabelling the branch points we may assume there are integers 
$h_0=0 < h_1 < \ldots < h_m=n$ such that $\Psi(c_i) = T_j$ if 
$h_{j-1} < i \le h_j$. Note that $X$ can be realised as the compact Riemann 
surface that completes the affine curve
\begin{equation*}
  \{ (x,y_1, \ldots, y_m) \in \cc^{m+1} \mid y_j^{d_j} = 
         (x-p_{h_{j-1}+1}) \cdots (x - p_{h_j}), \quad 
         1 \le j \le m \}
\end{equation*}
and $\pi(x,y_1, \ldots, y_m) = x$. As before if $n_j = h_j - h_{j-1}$, then
\begin{enumerate}[\tiny$\bullet$] \itemsep .5em
\item $\pi^{-1}(p_i)$ consists of $|\Gamma|/d_j$ points, each with the 
      ramification index $d_j$, if $h_{j-1} < i \le h_j$.
\item $\pi^{-1}(\infty)$ consists of $e$ points with ramification index
      $f = \ord(n_1T_1 + \ldots + n_m T_m)$ 
      where $e = |\Gamma|/f$.
\end{enumerate}
Then applying Riemann Hurwitz formula we obtain:
\begin{align*}
    g(X) \quad = \quad\frac{1}{2} \left( |\Gamma|(n-1) 
           - |\Gamma|\sum_{j=1}^m \frac{n_j}{d_j} - e \right) + 1.
\end{align*}

\subsection*{Representation of mixed braid group on $H^1(X)$}
For $b \in \B_n$, notice that
\begin{equation} \label{eq:invariance}
  \Psi(b_* c_i) = \Psi(c_i) \iff b \in \B_{n,\Lambda}.
\end{equation}
Hence every element of $\B_{n,\Lambda}$ lifts to $X$. It also 
follows from \eqref{eq:invariance} that for $b \in \B_{n,\Lambda}$ and any lift 
$\wt b : X \to X$ we have
\begin{equation*}
  \wt b \, \gamma = \gamma \, \wt b \tx{for all} \gamma \in \Gamma.
\end{equation*}
Since $b$ fixes a neighbourhood $N$ of $\infty$ in $\pp^1$ point wise, there is a
unique lift $\wt b$ which fixes $\pi^{-1}(N)$ in $X$. This gives rise to 
a homomorphism
\begin{equation} \label{eq:hom}
  \lift: \B_{n,\Lambda} \to \mcg(X)^{\Gamma}, \qquad b \mapsto \wt b,
\end{equation}
where $\mcg(X)^{\Gamma}$ is the centraliser of $\aut_{\pp^1} (X) \cong \Gamma$ in 
$\mcg(X)$.

The mapping class group $\mcg(X)$ acts on $H^1(X) := H^1(X, \cc)$ by
\begin{equation*}
  b \cdot \xi = \left(\wt b^{-1} \right)^* \xi, \tx{where}
  b \in \mcg(X) \text{ and } \xi \in H^1(X).
\end{equation*}
Thus the homomorphism \eqref{eq:hom} induces a representation of $\B_{n,\Lambda}$ on
$H^1(X)$. By a standard argument involving lifts of isotopies this 
is the same as the monodromy representation $\theta$ defined in 
$\eqref{eq:monodromy}$. 

Let $\mu(d_j)$ be the set of $d_j$-th roots of unity and 
$\rho = (\rho_1, \ldots, \rho_m) \in \mu(d_1) \times \ldots \times \mu(d_m)$.
For a $\cc[\Gamma]$ module $V$ we define the $\rho$ eigenspace of $\Gamma$ in
$V$ to be the subspace
\begin{equation*}
  V_\rho := \{ v \in V \mid T_j v = \rho_j v \ , \quad 1 \le j \le m\}.
\end{equation*}
Then $\mcg(X)^\Gamma$ preserves the eigenspaces $H^1(X)_\rho$ hence so does 
$\B_{n,\Lambda}$.

The standard hermitian form on $H^1(X)$ given by
\begin{equation} \label{eq:form}
  \inner{ \omega_1, \omega_2} = \int_{X} \omega_1 \wedge \ov\omega_2, \qquad
  \omega_1, \omega_2 \in H^1(X)
\end{equation}
is invariant under the $\mcg(X)$ action. Thus if $(r,s)$ is the signature of 
$\inner{\cdot, \cdot}$ on $H^1(X)_\rho$ the representation of $\B_{n,\Lambda}$ on 
$H^1(X)_\rho$ yields a homomorphism
\begin{equation*}
  \theta_\rho: \B_{n,\Lambda} \to \U(r,s).
\end{equation*}

%==================Third Section=====================================================

\section{Eigenspace Decomposition of the Cohomology}

In this section, we investigate the $\Gamma$-eigenspaces $H^1(X)_\rho$. 
For any orientable surface $S$ we have an isomorphism $\eta: H_1(S) \to H^1_c(S)$ 
induced by Poincar\'e duality, (see \cite[Chapter 1, \S 5]{BT} and 
\cite[\S III.1.1]{FK}), such that for any $x \in H_1(S)$ and any $\omega \in H^1(S)$
\begin{equation*}
  \int_x \omega = \int_S \omega \wedge \eta(x).
\end{equation*}
Moreover, if $f: S \to S$ is a diffeomorphism then
\begin{equation*}
  \eta(f_* x) = (f^{-1})^* \eta(x).
\end{equation*}
When $x$ is represented by a simple closed curve, see \cite[\S II.3.3]{FK} for an
explicit description of $\eta(x)$. Thus the isomorphism
\begin{equation*}
  \eta: H_1(X) \to H^1(X)
\end{equation*}
is in fact, an isomorphism of $\cc[\Gamma]$ modules, when $\Gamma$ acts on $H_1(X)$ 
via the push-forward
\begin{equation*}
  \gamma \cdot \xi = \gamma_* \xi \tx{for} \gamma \in \Gamma, \xi \in H_1(X),
\end{equation*}
while the action on $H^1(X)$ is given by the pull-back of the inverse:
\begin{equation*}
  \gamma \cdot \xi = (\gamma^{-1})^* \xi \tx{for} \gamma \in \Gamma, 
                                                  \xi \in H^1(X).
\end{equation*}
Consequently we may work with $H_1(X)$ or $H^1(X)$ interchangeably as far as the 
$\Gamma$ action is concerned.

\subsection*{Dimension}
Consider the CW structure on $\pp^1$ with $0$ cells: $ p_1, \ldots, p_n$ 
and $\infty$; $1$ cells: $E_1, \ldots, E_n$ where $E_i$ has end points $\infty$ 
and $p_i$; $2$ cells: $F$ with obvious attaching map.
This induces the following CW structure on $X$:
\begin{itemize} \itemsep 5pt
\item $0$ cells:
      \begin{itemize} \itemsep 5pt
      \item $\tilde p_i^{\,\gamma}$ , with
            $\gamma \in \Gamma_{j} := \dfrac{\Gamma}{(T_j)}$ 
            for $h_{j-1} < i \le h_j $,
      \item $\wt{\infty}^{\,\gamma}$ for 
            $\gamma \in \Gamma_\infty := 
            \dfrac{\Gamma}{\Big(n_1T_1 + \cdots + n_mT_m\Big)}$.
      \end{itemize}          
\item $1$ cells: $\wt{E}_i^{\gamma}$ for $\gamma \in \Gamma$. 
\item $2$ cells: $\wt{F}^\gamma$ for $\gamma \in \Gamma$.
\end{itemize}

The deck group $\Gamma$ acts on $X$ by permuting cells, so the cellular 
chain complex 
\begin{equation*}
  0 \to \C_2(X) \to \C_1(X)\to \mathcal{C}_0(X)\rightarrow 0 
\end{equation*}
is also a chain complex of $\cc[\Gamma]$ modules. It follows that 
\begin{equation*}
    \sum_{k=0}^{2} (-1)^k \dim H_k(X)_\rho =
    \sum_{k=0}^2 (-1)^k \dim \C_n(X)_\rho. 
\end{equation*}
We have
\begin{align*}
  \C_0(X) & \cong\cc[\Gamma_{\infty}] \oplus 
            \left( \bigoplus_{j=1}^m \cc[\Gamma_j]^{n_j} \right), &
  \C_1(X) & \cong \cc[\Gamma]^n, &
  \C_2(X) & \cong \cc[\Gamma].
\end{align*}
Let $\delta: \cc \to \{0,1\}$ be the function $\delta(1) = 1$ 
and $\delta(z) = 0$ if $z \neq 1$, then
\begin{equation*}
  \dim \C_0(X)_\rho = \sum_{j=1}^m \delta(\rho_j)n_j \ + \
                      \delta(\rho_1^{n_1} \cdots \rho_m^{n_m}).
\end{equation*}

Using the fact that $H_0(X) \text{ and } H_2(X)$ has both trivial 
$\cc[\Gamma]$ modules we get:
\begin{align*}
  \dim H_1(X)_\rho & =
  \begin{cases}
    0                        & \rho_1 = \ldots = \rho_m = 1, \\
    (n-1) - \dim C_0(X)_\rho & \text{otherwise}.
  \end{cases}  
\end{align*}
In particular if $\rho_j \neq 1$ for all $j$ we have
\begin{equation}
  \dim H_1(X)_\rho =
  \begin{cases}
    n-1, & \quad \rho_1^{n_1} \cdots \rho_m^{n_m} \neq 1,\\[10pt]
    n-2, & \quad \rho_1^{n_1} \cdots \rho_m^{n_m} = 1.
  \end{cases}
\end{equation}

\subsection*{Signature}
We have the Hodge decomposition $H^1(X)_\rho = H^{1,0}(X)_\rho \oplus H^{0,1}(X)_\rho$. 
The intersection form $\inner{\cdot,\cdot}$ is positive definite on $H^{1,0}(X)$ 
and negative definite on $H^{0,1}(X)_\rho$. Using the Chevalley-Weil formula
(\cite{CWH}, \cite[Theorem 3.1]{Can}) we can compute the dimension of 
$H^{1,0}(X)_\rho$.

For any Galois cover $\pi: Z \to W$ of compact Riemann surfaces, with Galois group 
$\Gamma = \aut_W(Z)$:
\begin{itemize} \itemsep .5em
\item Let $B(\pi) \subset W$ be the set of branch points of $\pi$.
\item For each $p \in B(\pi)$, pick any $q\in \pi^{-1}(p)$ and let 
      $\gamma_p \in \Gamma$ be a generator of the stabiliser of $q$
      and $\nu_p = \ord(\gamma_p)$ be the ramification index of $\pi$ at $q$.
\item For any representation $\chi$ of $\Gamma$ and $p\in B(\pi)$ let
      \begin{equation*}
        N_p(\chi) := \sum_{j=1}^{\nu_p-1}
                       \Big(\dim \ker \left(\chi(\gamma_p) 
                       - \zeta_{\nu_p}^j \id\right)\Big)
                     \left(1 - \frac{j}{\nu_p} \right).
      \end{equation*}      
\end{itemize}   
The multiplicity of an irreducible representation $\chi$ of $\Gamma$
in the $\mathbb{C}[\Gamma]$-module $H^{1,0}(Z)$ is given by 
\begin{equation*}
  \epsilon + (\dim \nu)(g(W) - 1) + \sum_{p \in B(\pi)} N_p(\chi)
\end{equation*}    
where $\epsilon = 1$ if $\chi$ is the trivial representation otherwise
it is $0$.

To apply this to our situation $\pi: X \to \pp^1$ note that our group 
$\Gamma = \aut_{\pp^1}(X)$ is Abelian so all irreducible representations are 
$1$-dimensional. Moreover, the branch points are $p_1, \ldots, p_n, \infty$ with
\begin{equation*}
  \gamma_{\infty} = n_1T_1 +\cdots+ n_mT_m \tx{and}
  \gamma_{p_i} = T_j \quad \tx{for} h_{j-1} < i \le h_j.
\end{equation*}
For any $\rho \in \mu(d_1) \times \cdots \times \mu(d_m)$ and let 
$\chi_\rho: \Gamma \to \cc^{\times}$ be the homomorphism $\chi_\rho(T_j) = \rho_j$.
Then $\dim H^{1,0}(X)_\rho$ is the multiplicity of $\chi_\rho$
in $H^{1,0}(X)$. Setting $\rho_j = \zeta_{d_j}^{-k_j}$ 
\begin{align*}
  \dim H^{1,0}(X)_{\rho} & = -1 + N_{\infty}(\chi_\rho) 
                              + \sum_{i=1}^n N_{p_i}(\chi_\rho) \\
      & = -1 + N_{\infty}(\chi_\rho) + \sum_{j=1}^m \frac{n_jk_j}{d_j}
        = \left\lceil \sum_{j=1}^m \frac{n_jk_j}{d_j} - 1\right\rceil,       
\end{align*}
since $0 < N_\infty(\chi_\rho) < 1$. Hence for 
$\rho = (\zeta_{d_1}^{-k_1}, \ldots, \zeta_{d_m}^{-k_m})$ the form 
$\inner{\cdot, \cdot}$ has signature $(r,s)$ on $H^1(X)_{\rho}$ where
\begin{align} \label{eq:signature}
  r & = \left\lceil \sum_{j=1}^m \frac{n_jk_j}{d_j} - 1\right\rceil, &
  s & = \left\lceil n - 1 - \sum_{j=1}^m \frac{n_jk_j}{d_j}\right\rceil \ .
\end{align} 

\subsection*{Subsurfaces}
We consider the following open disks in $\cc$ each containing exactly 
two branch points:
\begin{enumerate}[$\bullet$] \itemsep 5pt
\item Let $\D_i$ be a disk containing $p_i$ and $p_{i+1}$ and no other 
      branch points, where $1 \le i < n$, and $i\neq h_j$ for any $j$.
      
\item Let $\E_{j,k}$ be a disk containing $p_{h_j}$ and $p_{h_{k-1}+1}$ 
      and no other branch points, for $1 \leq j < k \le m$. That is $\E_{j,k}$
      contains the last point of the $j$-th part and first point of $k$-th part.
\end{enumerate}
\begin{figure}[h]
	\begin{center}
	  \includegraphics[height=5cm]{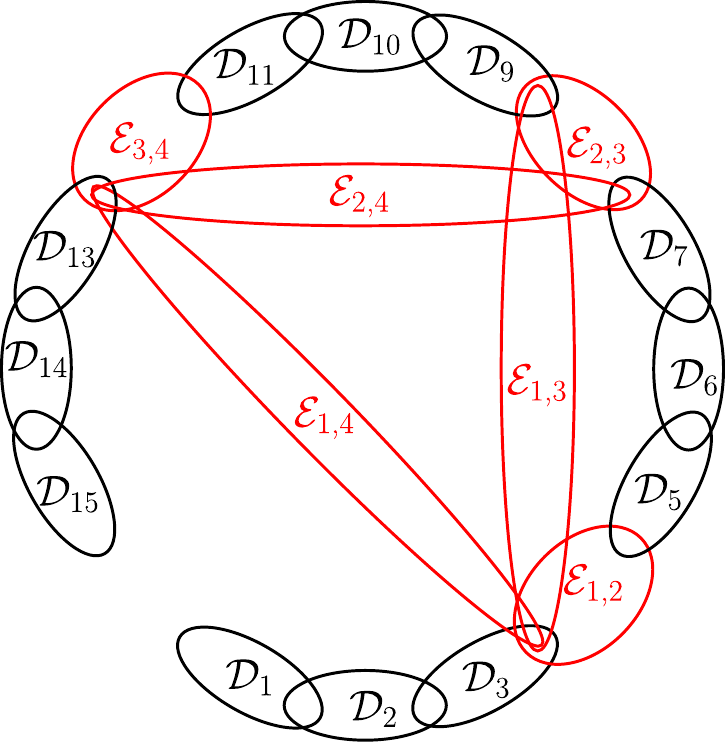}
	\end{center}
	\caption{Arrangement of the disks when $n=16$, $m=4$, and all parts have size $4$.}
\end{figure}
Let $\tD_i := \pi^{-1} (\D_i)$ and $\tE_{j,k} := \pi^{-1}(\E_{j,k})$.

For any group $G$ the augmentation ideal is defined as
\begin{equation*}
  I[G] := \left\{ \sum a_\gamma \gamma \in \cc[G] \mid  \sum a_\gamma = 0 \right\}.
\end{equation*}
As $\cc[G]$ modules we have an isomorphism
\begin{equation*}
  \cc[G] \cong \mathbf{1}_G \oplus I[G] \tx{where}
  \mathbf{1}_G \text{ is the trivial representation}.
\end{equation*}

\begin{proposition} \label{prop:subsurf}
  As $\cc[\Gamma]$ modules we have isomorphisms
  \begin{align*}
    H^1_c(\tD_i) & \cong \cc[\Gamma] \otimes_{\cc[(T_j)]} I[(T_j)], 
                 & h_{j-1} < i \le h_j, \\
    H^1_c(\tE_{j,k}) & \cong \cc[\Gamma] \otimes_{\cc[(T_j, T_k)]} I[(T_j, T_k)].
  \end{align*}
  Hence
  \begin{align*}
	  \dim H^1_c(\tD_i)_{\rho} & = 
	  \begin{cases}
	    0 & \rho_j = 1, \\
	    1 & \rho_j \neq 1,
	  \end{cases} &
	  \dim H^1_c(\tE_{j,k})_{\rho} & = 
	  \begin{cases}
	    0 & \rho_j = \rho_k = 1, \\
	    1 & \text{otherwise}.
	  \end{cases}
  \end{align*}
\end{proposition}

\begin{proof}
	Each component of $\tD_i$ is stabilised by $T_j$ if $h_j < i \le h_{j+1}$.
	Thus $\tD_i$ has $|\Gamma/(T_j)|$ components which we index
	\begin{equation*}
	  \tD_i^\gamma \text{ for } \gamma \in \frac{\Gamma}{(T_j)} \tx{ so that }
	  \kappa \tD_i^\gamma = \tD_i^{\kappa \gamma} 
	  \text{ for } \kappa ,\gamma \in \frac{\Gamma}{(T_j)}.
	\end{equation*}
	The restriction $\pi: \tD_i^\gamma \to \D_i$ is totally branched over 
	$p_i, p_{i+1}$ and has cyclic deck group $(T_j)\subset \Gamma$. Hence by 
	\cite[Theorem 4.4]{Mc} as a $\cc[(T_j)]$ module $H^1_c(\tD_i^\gamma)$ is isomorphic
	to $I[(T_j)]$. Therefore,
	\begin{equation*}
	  H^1_c(\tD_i) = \bigoplus_{\gamma \in \Gamma/(T_j)} H^1_c(\tD_i^\gamma) \quad
	                 \cong \quad \cc[\Gamma] \otimes _{\cc[(T_j)]} I[(T_j)].
	\end{equation*}	
	
	Turning our attention to $\tE_{j,k}$ we first note that each of its components   
	is stabilised by $(T_j, T_k)$. Hence, $\tE_{j,k}$ has $|\Gamma/(T_j, T_k)|$ 
	components which we index as
	\begin{equation*}
	  \tE_{j,k}^\gamma \text{ for } \gamma \in \frac{\Gamma}{(T_j, T_k)} \tx{ so that }
	  \kappa \tE_{j,k}^\gamma = \tE_{j,k}^{\kappa \gamma} 
	  \text{ for } \kappa ,\gamma \in \frac{\Gamma}{(T_j, T_k)}.
	\end{equation*}
	The restriction $\pi: \tE_{j,k}^\gamma \to \E_{j,k}$ is branched over 
	$p_{h_j}, p_{h_{k-1}+1}$ and has deck group $(T_j, T_k) \subset \Gamma$.	
	
	Let $Z$ be the compact Riemann surface which completes 
	\begin{equation*}
	  \{ (x,y) \in \cc^2 \mid y^{d_k} = x^{d_j}-1 \}. 
	\end{equation*} 
	The map $\delta: Z \to \pp^1$, $\delta(x,y) = x^{d_j}$ is branched over $0$, $1$ and 
	$\infty$. There is an isomorphism
	\begin{equation*}
	  \tE_{j,k}^\gamma \cong \delta^{-1} \{x \in \cc \mid |x| < 2\}.
	\end{equation*} 
	Under this identification the deck transformations act by 
	\begin{equation*}
	  T_j(x,y) = (\zeta_{d_j}x, y) \tx{and} T_k(x,y) = (x, \zeta_{d_k}y).
	\end{equation*}
	The long exact sequence in compactly supported cohomology 
	(see \cite[Proposition 13.11]{MT}) yields
	\begin{align*}
	  0 \to H^0(Z) \to H^0(Z - \tE_{j,k}^\gamma) 
	  & \to H^1_c(\tE_{j,k}^\gamma) \to H^1(Z) \to 0
	\end{align*} 
	since $Z - \tE_{j,k}^\gamma$ is a disjoint	union of disks.
	
	Let $H^1(Z)_{\rho_j, \rho_k} : = \ker(T_j - \rho_j\id) \cap \ker(T_k - \rho_k\id)$
	then by \cite[Theorem 3.1]{Mc}
	\begin{equation*}
	  \dim H^1(Z)_{\rho_j, \rho_k} =
	  \begin{cases}
	    1, & \rho_j\rho_k \neq 1, \\
	    0, & \rho_j\rho_k = 1.
	  \end{cases}
	\end{equation*}
	On the other hand $Z - \tE_{j,k}^\gamma$ is a disjoint union of disks each 
	containing a pre-image of $\infty$. Note that $\infty$ has $a=\gcd(d_j,d_k)$
	pre-images in $Z$ each stabilised by $T_j+T_k$. The function 
	\begin{equation*}
	  F: Z \to \pp^1 \tx{given by} F(x,y) = \frac{y^{d_k/a}}{x^{d_j/a}}
	\end{equation*}
	maps the pre-images of $\infty$ to the $a$-th roots of unity. Moreover
	\begin{equation*}
	  (T_j^{-1})^* F = \zeta_a F \tx{and} (T_k^{-1})^* F = \zeta_a^{-1} F.
	\end{equation*}
	Hence for $H^0(Z - \tE_{j,k}^\gamma)_{\rho_j,\rho_k} : 
	= \ker(T_j - \rho_j\id) \cap \ker(T_k - \rho_k\id)$ we have
	\begin{equation*}
	  \dim H^0(Z - \tE_{j,k}^\gamma)_{\rho_j,\rho_k} = 
	  \begin{cases}
	    1 & \rho_j\rho_k = 1, \\
	    0 & \rho_j\rho_k \neq 1.
	  \end{cases}
	\end{equation*}
	This shows $H^1_c(\tE_{j,k}^\gamma) \oplus H^0(Z) \cong \cc[(T_j,T_k)]$.
	Since $H^0(Z)$ is the trivial representation it follows that 
	$H^1_c(\tE_{j,k}^\gamma) \cong I[(T_j,T_k)]$. Finally
	\begin{equation*}
	  H^1_c(\tE_{j,k}) = 
	  \bigoplus_{\gamma \in \Gamma/(T_j,T_k)} H^1_c(\tE_{j,k}^\gamma)
	  \cong \cc[\Gamma] \otimes_{\cc[(T_j,T_k)]} I[(T_j,T_k)].
	\end{equation*}
\end{proof}

\subsection*{Generating Curves}
Let $\mu'(d_j) = \mu(d_j) -\{1\}$, then for $\rho \in \mu'(d_1) \times \cdots 
\times \mu'(d_m)$ we shall pick explicit elements of the
one dimensional eigenspaces $H^1_c(\tD_i)_{\rho}$ and $H^1_c(\tE_{j,k})_{\rho}$.

Let $s_i$ be the curve in the punctured disk $\D_i -\{p_i, p_{i+1}\}$ and $t_{j,k}$
be the curve in the punctured disk $\E_{j,k} -\{p_{h_j}, p_{h_{k-1}+1}\}$ as shown 
in Figure \ref{fig:gencurv}.
\begin{figure}[h]
\begin{center}
  \begin{tabular}{ccc}
  \includegraphics[height=4cm]{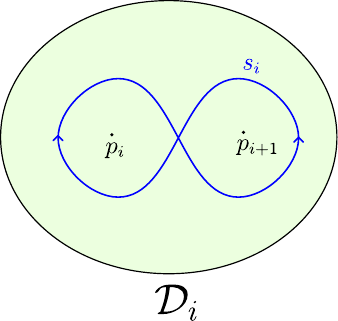} & \hspace{1cm} &
  \includegraphics[height=4cm]{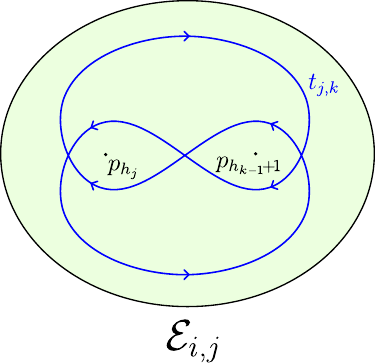}
  \end{tabular}
\end{center}
\caption{The curves $s_i$ and $t_{j,k}$.}
\label{fig:gencurv}
\end{figure}

Let $\epsilon_1, \epsilon_2$ be small counter-clockwise loops encircling 
$p_i$ and $p_{i+1}$ in $\D_i$ based at the center of the disk. Similarly let
$\delta_1, \delta_2$ be a small counter-clockwise loops encircling 
$p_{h_j}$ and $p_{h_{k-1}+1}$ in $\E_{j,k}$. Then
\begin{align*}
  s_i & = \epsilon_2\epsilon_1^{-1} 
      & \text{ in } \Pi_1(\D_i -\{p_i, p_{i+1}\}) \text{ and } \\
  t_{j,k} & = \delta_1\delta_2\delta_1^{-1}\delta_2^{-1} 
             & \text{ in } \Pi_1(\E_{j,k} -\{p_{h_j}, p_{h_{k-1}+1}\}).
\end{align*}
Hence $\Psi(s_i) = \Psi(t_{j,k}) = 0$ and consequently the lifts of $s_i$ and 
$t_{j,k}$ in $X$ are loops and in fact simple closed curves. 

We choose the curves $s_i$ and $t_{j,k}$ (upto homotopy) so that:
\begin{enumerate}[$\bullet$] \itemsep 5pt
\item \underline{\emph{Adjacent $s$-curves}}: If $i, i+1$ and $i+2$ are in the 
      same part of $\Lambda$, the curves $s_i$ and $s_{i+1}$ intersect at 
      two points; we designate one of these intersection points as
      $x_i$ (see Figure \ref{fig:int}).
    
\item \underline{\emph{Transition curves}}: If $n_j\geq 2$ the pair 
      $(s_{h_j-1}, t_{j,j+1})$ intersects in four points, we we distinguish 
      $x_{h_j-1}$. Similarly, when $n_{j+1} \geq 2$ we distinguish the point $x_{h_j}$      
      among the four intersection points of the pair$(t_{j,j+1}, s_{h_j+1})$.
    
\item \underline{\emph{Intersecting $t$-curves}}: If $n_j=1$, the pair
      $(t_{j-1,j}, t_{j,j+1})$ intersect in eight points we distinguish $x_{h_j-1}$.
      Similarly  the pairs $(t_{k-1,k}, t_{j,k})$ and $(t_{j,k}, t_{j,j+1})$ each 
      intersect at eight points we distinguish $y_{j,k}$ and $z_{j,k}$ respectively.
\end{enumerate}

\begin{figure}[h]
  \begin{center}
	  \includegraphics[width=.8\textwidth]{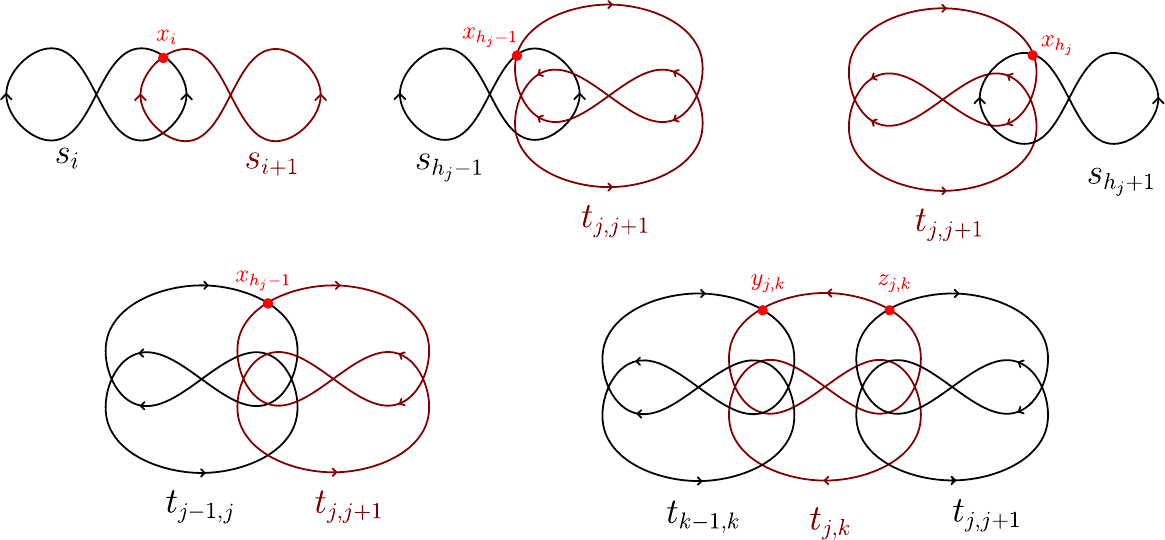} \\[10pt]
	  \includegraphics[width=.9\textwidth]{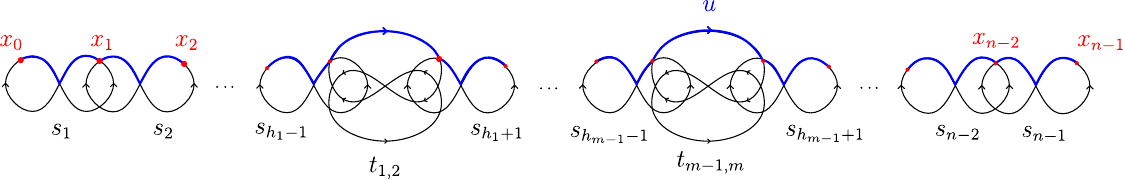}
  \end{center}
  \caption{The first two rows show the intersections of the generating curves. 
    The third row shows the curve $u$ in blue when all parts $\Lambda_j$ have size 
    at least 2.}
  \label{fig:int}   
\end{figure}
Fix points $x_0 \in s_1$ (or $t_{1,2}$ if $n_1=1)$ and $x_{n-1} \in s_{n-1}$ 
(or $t_{m-1,m}$ if $n_m=1$). Let $u$ be the path from $x_0$ 
to $x_{n-1}$ that traverses each point $x_i$. The curve $u$ is illustrated in the 
final row of Figure \ref{fig:int} when $n_1$ and $n_m$ are at least 2. In all other 
cases the curve is similar. We denote the lifts of $u$ by 
$\{\tilde{u}^\gamma\}_{\gamma \in \Gamma}$, indexed such that the group action 
satisfies:
\begin{equation*}
  \kappa \tilde{u}^\gamma = \tilde{u}^{\kappa + \gamma}.
\end{equation*}
For each $i \in \{0, \ldots, n-1\}$, let $\tilde{x}_i^\gamma$ be the preimage of 
$x_i$ situated on the lift $\tilde{u}^\gamma$. 
\begin{itemize} \itemsep 5pt
\item Let $\tilde{s}_i^{\,\gamma}$ be the lift of $s_i$ that originates at 
      $\tilde{x}_{i-1}^\gamma$.
\item Let $\tilde{t}_{j,j+1}^{\,\gamma}$ be the lift of $t_{j,j+1}$ that originates 
      at $\tilde{x}_{h_j-1}^\gamma$.
\end{itemize}
Finally, let $\tilde{y}_{j,k}^\gamma$ be the unique pre-image of $y_{j,k}$ contained 
in the lift $\tilde{t}_{k-1,k}^{\,\gamma}$ and define $\tilde{t}_{j,k}^{\,\gamma}$ 
to be the lift of $t_{j,k}$ starting at $\tilde{y}_{j,k}^\gamma$.

\begin{remark}
  The lifts $\left\{ \Ts_i^{\,kT_j} \mid 0 \le k < d_j\right\}$ span $H_1(\tD_i^0)$ 
  and are subject to a single linear relation
  \begin{equation*}
    \Ts_i^{\,0} + \Ts_i^{\,T_j} + \ldots + \Ts_i^{\,(d_j-1)T_j} = 0.
  \end{equation*}
  Similarly $\left\{\Tt_{j,k}^{\,\gamma} \mid \gamma \in (T_j, T_k)\right\}$ span 
  $H_1(\tE_{j,k}^0)$ and satisfy the linear relation
  \begin{equation*}
    \sum_{q=1}^{d_j} \sum_{r=1}^{d_k} \Tt_{j,k}^{\,qT_j + rT_k} = 0. 
  \end{equation*}
  These assertions follow easily from Corollary \ref{cor:span}.
\end{remark}

For $\rho \in \mu'(d_1) \ldots \times \mu'(d_m)$ we define the cohomology classes 
$\omega_i(\rho) \in H^1_c(\tD_i)_\rho$, 
$\omega_{h_j}(\rho) \in H^1_c(\tE_{j,j+1})_\rho$ and 
$\phi_{j,k}(\rho) \in H^1_c(\tE_{j,k})_\rho$ when $j+1<k$ via the following 
weighted sums:
\begin{equation} \label{eq:gens}
  \begin{aligned} 
	  \omega_i(\rho) & = \frac{\sqrt{2}}{|\Gamma|^{1/2}(\rho_j-1)} \
	      \sum_{\gamma \in \Gamma} \ov\rho^\gamma \, \eta(\tilde{s}_i^{\,\gamma}) 
	      & 1 \le i \le n-1, \ i \neq h_j, \\   
	  %--------------------------------------------------------------    
	  \omega_{h_j}(\rho) & = \frac{\sqrt{2}}{|\Gamma|^{1/2}(\rho_j-1)(1-\ov\rho_{j+1})} \
	    \sum_{\gamma \in \Gamma} \ov\rho^\gamma \, \eta(\tilde{t}_{j,j+1}^{\,\gamma})
	    & 1 \le j < m ,\\
	  %--------------------------------------------------------------    
	  \phi_{j,k}(\rho) & = \frac{\sqrt{2}}{|\Gamma|^{1/2}(1-\ov\rho_j)(\rho_k-1)} \
	    \sum_{\gamma \in \Gamma} \ov\rho^\gamma \, \eta(\tilde{t}_{j,k}^{\,\gamma}) 
	    & 1 \le j, \ j+1 < k\le m.
  \end{aligned}
\end{equation}
In these expressions, we employ the multi-index notation 
$\ov\rho^\gamma = \ov\rho_1^{\gamma_1} \cdots \ov\rho_m^{\gamma_m}$ for 
$\rho = (\rho_1, \ldots, \rho_m)$ and 
$\gamma = (\gamma_1, \ldots, \gamma_m) \in \Gamma$. 

In the following section, we show that the classes $\omega_i(\rho)$ span the 
eigenspaces $H^1(X)_\rho$. For a first reading, the scalar coefficients 
preceding the summations may be disregarded; they are normalized specifically 
to ensure the intersection numbers in Lemma \ref{lem:intomega} and 
Lemma \ref{lem:intphi} are uniform.

%==================Fourth Section====================================================

\section{The Hermitian Intersection Form}

Here we shall calculate intersections among the classes defined 
in \eqref{eq:gens}. All intersections are calculated using the basic fact that
for simple closes curves $k, k'$ in $X$
\begin{equation*}
  \langle \eta(k), \eta(k') \rangle = \frac{\sqrt{-1}}{2} (k\cdot k') \ ,
\end{equation*}
where $(k\cdot k')$ is the signed intersection number of the curves.

Let us fix $ \rho = (\rho_1, \ldots, \rho_m) $. Then we write 
$\omega_i$ instead of $\omega_i(\rho)$ and $\phi_{j,k}$ instead of $\phi_{j,k}(\rho)$.

\begin{lemma} \label{lem:intomega}
  Let $1 \le i < m$, then
  \begin{equation*} \renewcommand{\arraystretch}{3}
	  \begin{array}{lcccr} 
		  \inner{\omega_i, \omega_i}     
		    & =  & \sqrt{-1} \left( \dfrac{1+\rho_j}{1-\rho_j} \right)
		    & \qquad & h_{j-1} < i < h_j, \\
		  %--------------------------------------------------------------------------
	    \inner{\omega_{h_j}, \omega_{h_j}} 
	      & = & \sqrt{-1} \left(\dfrac{1-\rho_j\rho_{j+1}}
                                     {(1-\rho_j)(1-\rho_{j+1})} \right)         
	          & & 1 \le j < m, \\
		  %--------------------------------------------------------------------------
	    \inner{\omega_{i-1}, \omega_i} 
	      & = & \sqrt{-1}\left(\dfrac{ -\rho_j}{1-\rho_j} \right) 
	      & & h_{j-1} < i \le  h_j \\
	  \end{array}              
  \end{equation*}
  and $\inner{\omega_s, \omega_t} = 0$ if $|t-s| > 1$.
\end{lemma}

\begin{remark}
  We note that 
  \begin{align*}
    \sqrt{-1} \left( \dfrac{1+\rho_j}{1-\rho_j} \right) 
      & = \dfrac{\im \rho_j}{\re \rho_j-1} \quad \text{and} \\[10pt]
    \sqrt{-1} \left(\dfrac{\rho_j\rho_{j+1} - 1}
                          {(\rho_j-1)(\rho_{j+1}-1)} \right)
      & = \dfrac{\im (\rho_j\rho_{j+1} -\rho_j - \rho_{j+1})}
	          {2(1- \re \rho_j)(1 - \re \rho_{j+1})}.
  \end{align*}
  are both real numbers.
\end{remark}

\begin{proof} 
  \begin{enumerate}[(a)] \itemsep 20pt
  \item \textbf{Calculation of $\inner{\omega_i, \omega_i}$ when $i \neq h_j$:} 
        \\[.5em]
        Let $w$ be the self-intersection point of $s_i$ and let $\tilde w^\gamma$ 
        be the pre-image of $w$ that lies on $\tilde u^\gamma$. 
        The lift $\Ts_i^{\,\gamma}$ passes through two distinct pre-images of 
        $w$, namely $\tilde w^\gamma$ and $\tilde w^{\gamma+T_j}$. Consequently, the 
        only non-trivial intersections between different lifts are:
        \begin{itemize}
        \item $\Ts_i^{\,\gamma} \cdot \Ts_i^{\,\gamma-T_j} = -1$ 
              (at $\tilde w^\gamma$), and
        \item $\Ts_i^{\,\gamma} \cdot \Ts_i^{\,\gamma+T_j} = 1$ 
              (at $\tilde w^{\gamma+T_j}$).
        \end{itemize}
	      \begin{center}
	        \includegraphics[width=.8\textwidth]{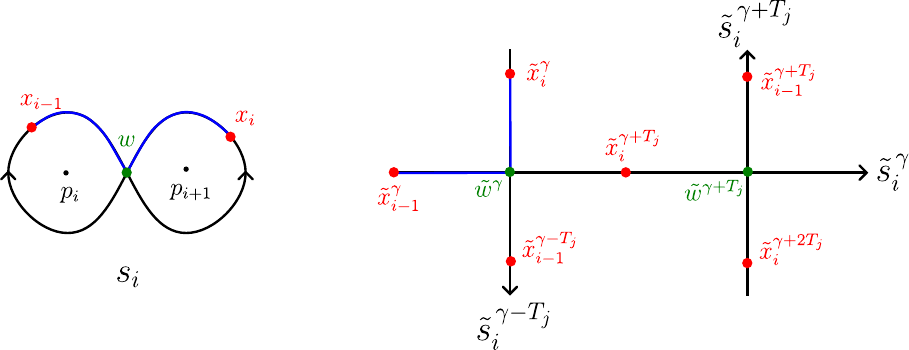}
	      \end{center}
	      Thus
	      \begin{align*}
	        \inner{\omega_i, \omega_i} 
	        & = \frac{\sqrt{-1}}{|\Gamma|(\rho_j-1)(\ov\rho_j-1)}
	          \sum_{\gamma, \kappa \in \Gamma} \ov\rho^{\gamma - \kappa}
	          (\tilde{s}_i^{\,\gamma} \cdot \tilde{s}_i^{\,\kappa}) \\
	        & = \frac{\sqrt{-1}}{|\Gamma|(\rho_j-1)(\ov\rho_j-1)}
	          \sum_{\gamma \in \Gamma} \left[
	          \rho_j(\tilde{s}_i^{\,\gamma} \cdot \tilde{s}_i^{\,\gamma + T_j}) + 
	          \ov \rho_j(\tilde{s}_i^{\,\gamma} \cdot \tilde{s}_i^{\,\gamma - T_j})
	          \right]
	         = \frac{\sqrt{-1}(\rho_j - \ov \rho_j)}{(\rho_j-1)(\ov\rho_j-1)}. 
	      \end{align*}
	      For future reference we also note that $\Ts_i^{\,\gamma}$ passes through
	      $\tilde x_i^{\gamma+T_j}$. This because $\tilde{u}^\gamma$ switches
	      over from $\Ts_i^{\,\gamma}$ to $\Ts_i^{\,\gamma-T_j}$ at $\tilde w^\gamma$. 
	      Hence $\tilde x_i^{\gamma}$ lies on $\Ts_i^{\,\gamma-T_j}$ rather than
	      $\Ts_i^{\,\gamma}$.
	      
	\item \textbf{Calculation of $\inner{\omega_{h_j}, \omega_{h_j}}$:} \\[.5em]	
	      Arguments are similar to the previous case.    
	      \begin{center}
	        \includegraphics[width=.8\textwidth]{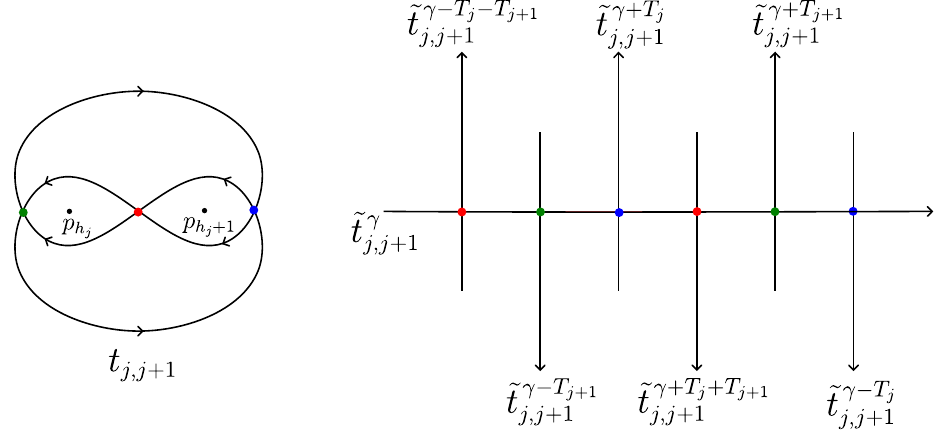}
	      \end{center}
	      All non-trivial intersections $\Tt_{j,j+1}^{\,\gamma} \cdot 
	      \Tt_{j,j+1}^{\,\gamma+\kappa}$ are listed in the table below:
	      \begin{align*} \renewcommand{\arraystretch}{2}
	        \begin{array}{|c|c|}
	          \hline
	          \kappa & \Tt_{j,j+1}^{\,\gamma} \cdot \Tt_{j,j+1}^{\,\gamma+\kappa} \\ 
	          \hline
	          T_j \ , \ T_{j+1} \ , \ -T_j-T_{j+1} & 1 \\ 
	          \hline
	          -T_j \ , \ -T_{j+1} \ , \ T_j+T_{j+1} & -1 \\ 
	          \hline
	        \end{array}
	      \end{align*}
	      Thus
	      \begin{align*}
	        \inner{\omega_{h_j}, \omega_{h_j}} 
	          & =  \frac{\sqrt{-1}(\ov\rho_j\ov\rho_{j+1} + \rho_j + \rho_{j+1}  
	                         - \rho_j\rho_{j+1} - \ov\rho_j - \ov\rho_{j+1})}
	                    {|1-\rho_j|^2|1- \rho_{j+1}|^2} \\[5pt]
	          & = \sqrt{-1}\frac{(1-\rho_j\rho_{j+1})(\ov\rho_j - 1)(\ov\rho_{j+1}-1)}
	                   {|1-\rho_j|^2|1- \rho_{j+1}|^2}
	            = \sqrt{-1}\dfrac{(1-\rho_j\rho_{j+1})}
	                             {(\rho_j-1)(\rho_{j+1}-1)}.
	      \end{align*}
	      
	\item \textbf{Calculation of $\inner{\omega_i, \omega_{i+1}}$ 
	      when $h_{j-1} < i < h_j-1$:} \\[.5em] 
	      The curves $s_i$ and $s_{i+1}$ intersect at $x_i$ and one other point
	      which we call $w$. From part (a) it follows that $\Ts_i^{\,\gamma}$ 
	      intersects $\Ts_{i+1}^{\,\gamma+T_j}$ at $\tilde x_i^{\gamma+T_j}$. 
	      \begin{center}
	        \includegraphics[width=.8\textwidth]{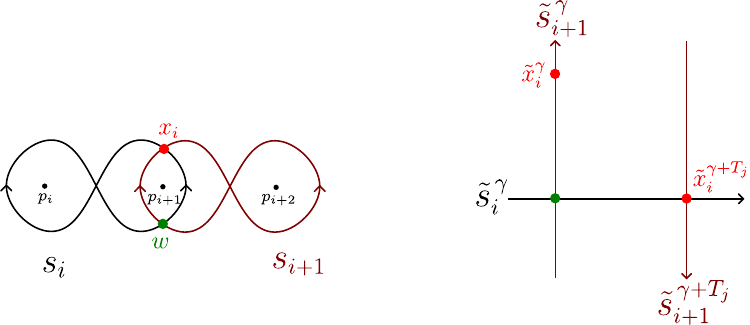}
	      \end{center} 
	      Consider the curve in $\cc$ which goes from $x_i$ to $w$ along $\bar{s}_i$ 
	      and comes back to $x_i$ along $s_{i+1}$. This is a clockwise
	      loop around $p_{i+1}$. Hence the lift of this loop starting at 
	      $\tilde x_i^{\gamma+T_j}$ ends at $\tilde x_i^\gamma$. Therefore 
	      $\Ts_i^{\,\gamma}$ has to intersect $\Ts_{i+1}^{\,\gamma}$ at a 
	      pre-image of $w$. Thus we have
	      \begin{equation*}
	        \Ts_i^{\,\gamma}\cdot \Ts_{i+1}^{\,\gamma} = 1, \qquad
	        \Ts_i^{\,\gamma}\cdot \Ts_{i+1}^{\,\gamma + T_j} = -1 \tx{and}
	        \Ts_i^{\,\gamma}\cdot \Ts_{i+1}^{\,\kappa} = 0 \text{ for }
	        \kappa \neq \gamma, \gamma + T_j,
	      \end{equation*}
	      which yields 
	      \begin{equation*}
	        \inner{\omega_i, \omega_{i+1}} = 
	        \sqrt{-1}\dfrac{(1 - \rho_j)}{(\rho_j-1)(\ov\rho_j-1)}
	        = \dfrac{\sqrt{-1}}{1-\ov\rho_j}.
	      \end{equation*}
	
	\item \textbf{Calculation of $\inner{\omega_{h_j-1}, \omega_{h_j}}$ 
	      when $h_{j-1}<h_j-1$:} \\[.5em]
	      The curves $s_{h_j-1}$ and $t_{j,j+1}$ intersect at $x_{h_j-1}$ and 
	      three more points which we call $w_1$, $w_2$ and $w_3$. From part $(a)$ 
	      we know that $\Ts_{h_j-1}^{\,\gamma}$ intersects $\Tt_{j,j+1}^{\,\gamma+T_j}$ 
	      at $\tilde{x}_{h_j-1}^{\gamma+T_j}$.
	      \begin{center}
	        \includegraphics[width=.9\textwidth]{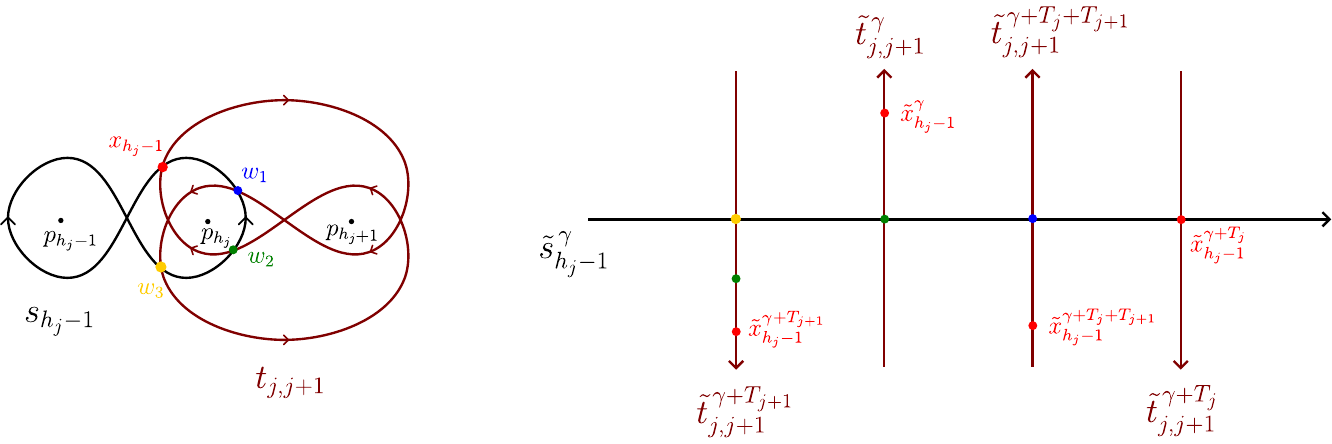}
	      \end{center} 
	      Consider the curve which goes from $x_{h_j-1}$ along $\ov{s}_{h_j-1}$ to 
	      $w_1$ and then back to $x_{h_j-1}$ along $\ov{t}_{j,j+1}$. This is a
	      counter-clockwise loop around $p_{h_j+1}$. Hence the lift of this loop 
	      starting at $\tilde{x}_{h_j-1}^{\gamma+T_j}$ ends at 
	      $\tilde{x}_{h_j-1}^{\gamma+T_j+T_{j+1}}$. Therefore $\Ts_{h_j-1}^{\,\gamma}$
	      must intersect $\Tt_{j,j+1}^{\,\gamma+T_j+T_{j+1}}$ at a pre-image of
	      $w_1$. Similarly we determine all the intersections. The non-trivial 
	      intersections are:
	      \begin{align*}\renewcommand{\arraystretch}{2}
	        \begin{array}{|c|c|}
	          \hline
	          \kappa & \Ts_{h_j-1}^{\,\gamma} \cdot \Tt_{j,j+1}^{\,\gamma+\kappa} \\ 
	          \hline
	          0 \ , \ T_j+T_{j+1} & 1 \\  \hline
	          T_j \ , \ T_{j+1} & -1 \\ \hline
	        \end{array}
	      \end{align*} 
	      Hence,
	      \begin{align*}
	        \inner{\omega_{h_j-1}, \omega_{h_j}} = 
	        \frac{\sqrt{-1} (1 + \rho_j \rho_{j+1} - \rho_j - \rho_{j+1})}
	             {(\rho_j-1)(\ov\rho_j-1)(1-\rho_{j+1})}  =
	        \frac{\sqrt{-1}(1 - \rho_j)(1 - \rho_{j+1})}
	             {(\rho_j-1)(\ov\rho_j-1)(1-\rho_{j+1})}.
	      \end{align*} 
	      
	\item \textbf{Calculation of $\inner{\omega_{h_j}, \omega_{h_j+1}}$ 
	      when $h_j+1 < h_{j+1}$:} \\[.5em]
	      Again the curves $t_{j,j+1}$ and $s_{h_j+1}$ intersect at $x_{h_j}$ and 
	      three more points which we call $w_1$, $w_2$ and $w_3$. We know that 
	      $\Tt_{j,j+1}^{\,\gamma}$ intersects $\Ts_{h_j+1}^{\,\gamma}$ at 
	      $\tilde{x}_{h_j}^\gamma$. 
	      \begin{center}
	        \includegraphics[width=.9\textwidth]{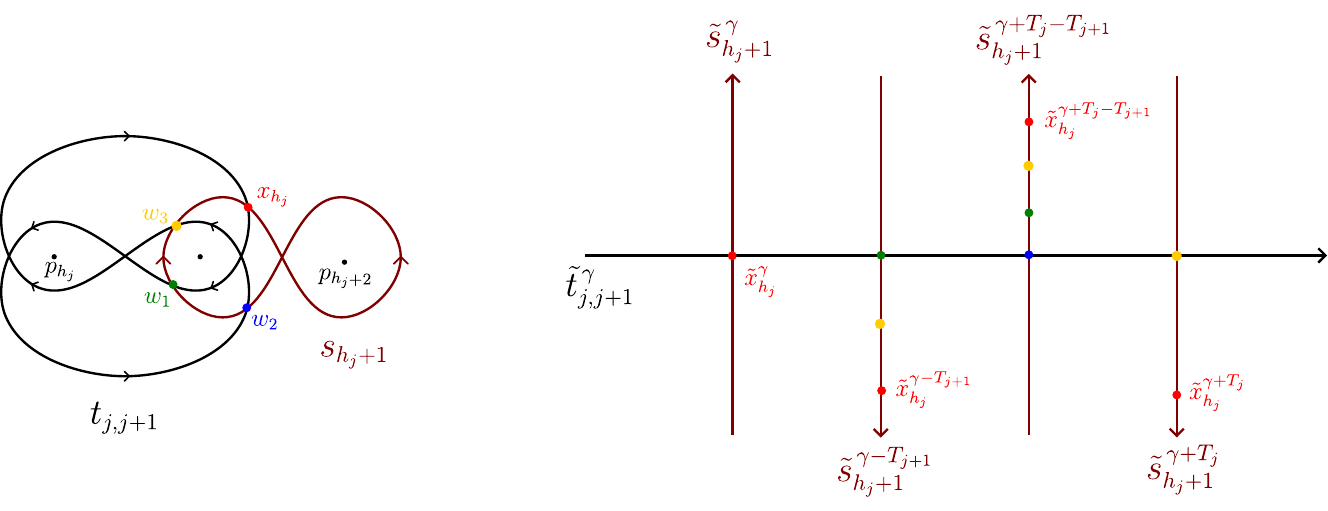}
	      \end{center}
	      By arguments similar the previous part we see that the non trivial 
	      intersections in this case are:
	      \begin{align*}
	        \renewcommand{\arraystretch}{2}
	        \begin{array}{|c|c|}
	          \hline
	          \kappa & \Tt_{j,j+1}^{\,\gamma}\cdot\Ts_{h_j+1}^{\,\gamma+\kappa} \\ 
	          \hline \hline
	          0 \ , \ T_j-T_{j+1} & 1 \\  \hline
	          T_j \ , \ -T_{j+1} & -1 \\ \hline
	        \end{array}
	      \end{align*} 
	      and $\Tt_{j,k}^{\,\gamma} \cdot \Ts_{h_j+1}^{\,\kappa} = 0$ in all other 
	      cases. Hence
	      \begin{align*}
	        \inner{\omega_{h_j}, \omega_{h_j+1}} & =
	        \frac{\sqrt{-1} ( 1 - \rho_j - \ov\rho_{j+1} + \rho_j \ov\rho_{j+1})}
	             {(\rho_j-1)(1 - \ov\rho_{j+1})(\rho_{j+1}-1)} \\[10pt]
	          & =\frac{\sqrt{-1}(1-\rho_j)(1 - \ov\rho_{j+1})}
	                  {(\rho_j-1)(1 - \ov\rho_{j+1})(\ov\rho_{j+1}-1)} =
	        \frac{\sqrt{-1}}{1 - \ov\rho_{j+1}}.
	      \end{align*}
	
	\item[(e')] \textbf{Calculation of $\inner{\omega_{h_j}, \omega_{h_j+1}}$ 
	      when $h_j+1 = h_{j+1}$:}  \\[.5em]
	      It is easy to see that $\Tt_{j,j+1}^{\,\gamma}$ intersects 
	      $\Tt_{j,j+1}^{\,\gamma}$ at $\tilde{x}_{h_j}^\gamma$.
				\begin{center}
				  \includegraphics[width=.9\textwidth]{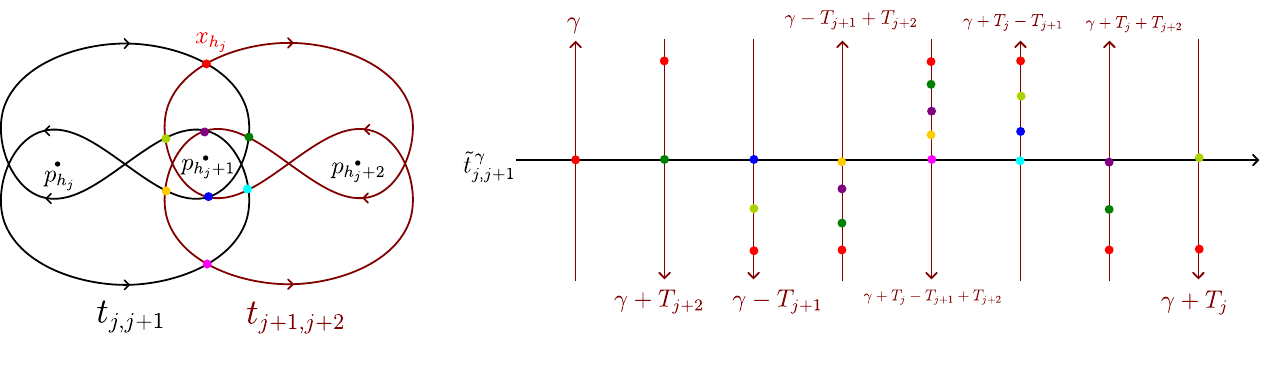}
				\end{center}
				In this case the non-trivial intersections are:
				\begin{align*}
	        \renewcommand{\arraystretch}{2}
	        \begin{array}{|c|c|}
	          \hline
	          \kappa & \Tt_{j,j+1}^{\,\gamma}\cdot\Tt_{j,j+2}^{\,\gamma+\kappa} \\ 
	          \hline \hline
	          0 \ , \ T_j-T_{j+1} \ , \  T_j + T_{j+2} \ , \  - T_{j+1}+T_{j+2} & 1 \\  
	          \hline
	          T_j \ , \ -T_{j+1} \ , \ T_{j+2} \ , \ T_j - T_{j+1}+T_{j+2} & -1 \\ 
	          \hline
	        \end{array}
	      \end{align*} 
				This yields
				\begin{align*}
				  \inner{\omega_{h_j}, \omega_{h_{j+1}}} & = 
				  \frac{\sqrt{-1}(1 - \rho_j - \ov \rho_{j+1} - \rho_{j+2} 
				        +\rho_j\ov\rho_{j+1} + \rho_j\rho_{j+2} 
				        + \ov\rho_{j+1}\rho_{j+2} - \rho_j\ov\rho_{j+1}\rho_{j+2})}
				       {(\rho_j-1)(1-\ov\rho_{j+1})(\ov\rho_{j+1}-1)(1 - \rho_{j+2})} \\
				    & = \frac{\sqrt{-1}(1 - \rho_j)(1-\ov\rho_{j+1})(1 - \rho_{j+2})}
				             {(\rho_j-1)(1-\ov\rho_{j+1})(\ov\rho_{j+1}-1)(1 - \rho_{j+2})}
				      = \frac{\sqrt{-1}}{1 - \ov\rho_{j+1}}.   
				\end{align*}
  \end{enumerate}
  Finally if $|t-s|>1$ the supports of $\omega_t$ and $\omega_s$ are disjoint.
\end{proof}

\begin{lemma} \label{lem:intphi}
  Let $1 \le j$ and $j+1 < k \leq m$ then:
  \begin{equation*} \renewcommand{\arraystretch}{3}
	  \begin{array}{ccc} 
	    \multicolumn{3}{c}{\inner{\phi_{j,k}, \phi_{j,k}} 
	      = \sqrt{-1} \left(\dfrac{1 -\rho_j\rho_k}
	            {(1-\rho_j)(1-\rho_k)}\right)} \\
		  %--------------------------------------------------------------------------
	    \inner{\omega_{h_{k-1}}, \phi_{j,k}} 
	      = \sqrt{-1} \left(\dfrac{ \rho_k}{1-\rho_k} \right)
	    & \qquad &
		  %--------------------------------------------------------------------------
	    \inner{\omega_{h_{k-1}+1}, \phi_{j,k}} 
	      = \sqrt{-1}\left(\dfrac{-1}{1-\rho_k} \right) \\
		  %--------------------------------------------------------------------------
	    \inner{\omega_{h_j},\phi_{j,k}}= \sqrt{-1} \left(\dfrac{1}{1-\rho_j} \right)
	    & \qquad &
		  %--------------------------------------------------------------------------
	    \inner{\omega_{h_j-1}, \phi_{j,k}} 
	    = \sqrt{-1}\left(\dfrac{-\rho_j}{1-\rho_j}\right)
	  \end{array}              
  \end{equation*}
  and in all other cases $\inner{ \phi_{j,k}, \omega_i} = 0$.
\end{lemma}

\begin{proof}
  \begin{enumerate}[(a)] \itemsep 20pt
  \item \textbf{Computation of $\inner{\phi_{j,k}, \phi_{j,k}}$:} 
        Similar to proof of Lemma \ref{lem:intomega} (b).
  
  \item \textbf{Computation of $\inner{\omega_{h_{k-1}}, \phi_{j,k} }$:}
        Note that $\Tt_{k-1,k}^{\,\gamma}$ intersects $\Tt_{j,k}^{\,\gamma}$ at
        $\tilde y_{j,k}^\gamma$. All other intersections can be determined from this, 
        so we simply list them and omit the repetitive arguments.
	      \begin{center}
	        \includegraphics[width=.9\textwidth]{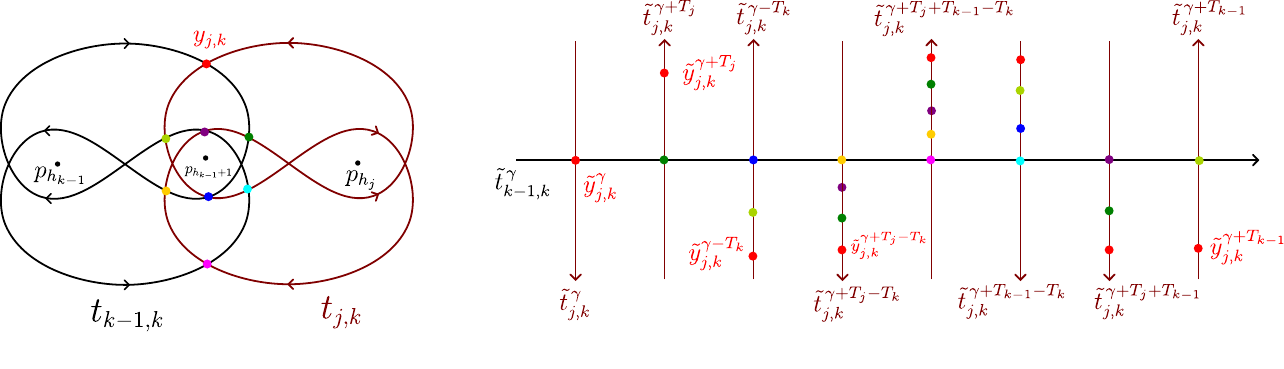}
	      \end{center}
	      \begin{equation*} \renewcommand{\arraystretch}{2}
	        \begin{array}{|c|c|} \hline
	          \kappa & \Tt_{k-1,k}^{\,\gamma} \cdot \Tt_{j,k}^{\,\gamma+\kappa} \\ 
	          \hline \hline
	          0 \ , \ T_j+T_{k-1} \ , \ T_j-T_k \ , \ T_{k-1}-T_k & -1 \\ \hline
	          T_j \ , \ T_{k-1} \ , \ - T_k \ , \ T_j + T_{k-1}-T_k & 1 \\ \hline
	        \end{array}
	      \end{equation*}
	      Hence
	      \begin{align*}
	        \inner{\omega_{h_{k-1}}, \phi_{j,k}} 
	          & = \frac{\sqrt{-1} (\rho_j\rho_{k-1} \ov\rho_k  
	                    + \rho_j + \rho_{k-1} + \ov\rho_k 
	                    - \rho_j \rho_{k-1} - \rho_j \ov\rho_k 
	                    - \rho_{k-1}\ov\rho_k - 1)}
	                   {(\rho_{k-1}-1)(1-\ov\rho_k)(\rho_j-1)(1-\ov\rho_k)} \\
	          & = \frac{\sqrt{-1}(\rho_j-1)
	                             (\rho_{k-1}-1)(\ov\rho_k-1)}
	                   {(\rho_{k-1}-1)(1-\ov\rho_k)(\rho_j-1)(1-\ov\rho_k)}
	            = \frac{\sqrt{-1}}{\ov\rho_k-1}.    
	      \end{align*}
	      
	\item \textbf{Computation of $\inner{\omega_{h_j}, \phi_{j,k} }$:} 
	      Let $\tilde z_{j,k}^\gamma$ denote the pre-image of $z_{j,k}$ lying on 
	      $\Tt_{j,j+1}^{\,\gamma}$. Consider the loop at $z_{j,k}$
	      formed by the portion of $u$ between $z_{j,k}$ and $y_{j,k}$ and the 
	      portion of $\ov t_{j,k}$  from $y_{j,k}$ back to $z_{j,k}$.
	      This loop is null homotopic in $\cc - \{p_1, \ldots, p_n\}$, hence 
	      $\Tt_{j,k}^{\,\gamma}$ intersects $\Tt_{j,j+1}^{\,\gamma}$ at
	      $\tilde{z}_{j,k}^\gamma$.
	      \begin{center}
	        \includegraphics[width=.9\textwidth]{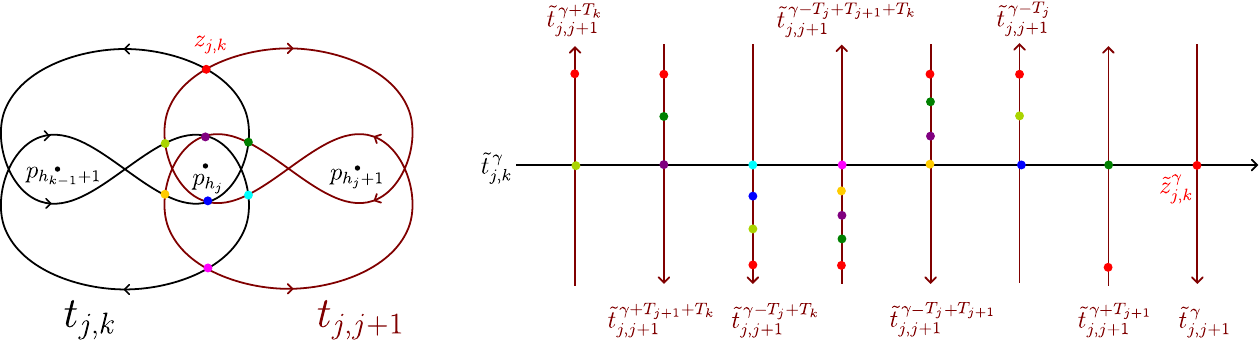}
	      \end{center}
	      Rest of the intersections can be determined as usual.
	      \begin{equation*} \renewcommand{\arraystretch}{2}
	        \begin{array}{|c|c|} \hline
	          \kappa & \Tt_{j,k}^{\,\gamma} \cdot \Tt_{j,j+1}^{\,\gamma+\kappa} 
	          \\ \hline \hline
	          0 \ , \ -T_j+T_{j+1} \ , \ -T_j+T_k \ , \ +T_{j+1}+T_k & -1 \\ \hline
	          -T_j \ , \ T_{j+1} \ , \ T_k \ , \ -T_j + T_{j+1} + T_k & 1 \\ \hline
	        \end{array}
	      \end{equation*}
	      Hence 
	      \begin{align*}
	        \inner{\phi_{j,k}, \omega_{h_j}}
	          & = \frac{\sqrt{-1}(\ov\rho_j\rho_k\rho_{j+1} + \ov\rho_j + \rho_{j+1} 
	                    + \rho_k - \ov\rho_j \rho_{j+1} 
	                    - \ov \rho_j\rho_k - \rho_{j+1} \rho_k - 1)}
	                   {(\ov\rho_j-1)(1-\rho_k)(\ov\rho_j-1)(1 -\rho_{j+1})} \\
	          & = \frac{\sqrt{-1}(\ov\rho_j-1)(\rho_k-1)(\rho_{j+1}-1)}
	                   {(\ov\rho_j-1)(1-\rho_k)(\ov\rho_j-1)(1 -\rho_{j+1})} 
	            = \frac{\sqrt{-1}}{\ov\rho_j-1}.
	      \end{align*}
	      
	\item \textbf{Computation of $\inner{\omega_{h_j-1}, \phi_{j,k} }$ 
	      when $h_{j-1} < h_j-1$:}
	      The curves $t_{j,k}$ and $s_{h_j-1}$ intersect at $4$ points and let $w$
	      be the intersection point shown in the figure below. 
	      \begin{center}
	        \includegraphics[width=.9\textwidth]{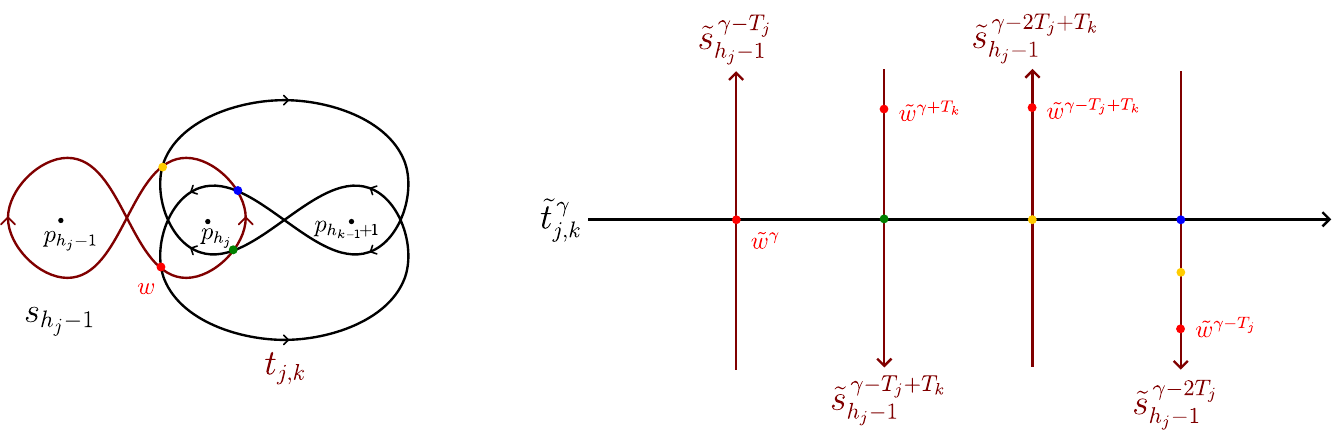}
	      \end{center}
	      Consider the triangle $\Delta$ with vertices $z_{j,k}$, $x_{h_j-1}$ and 
	      $w$ as shown in the figure below (upto homotopy). Since $\Delta$ 
	      is null-homotopic in $\cc - \{p_1, \ldots, p_n\}$, its lift 
	      $\wt\Delta^\gamma$ starting at $\tilde z_{j,k}^\gamma$ is again a triangle 
	      with vertices 
	      $\tilde z_{j,k}^\gamma$, $\tilde x_{h_j-1}^\gamma$ and a pre-image say 
	      $\tilde w^\gamma$ of $w$. Since $\Ts_{h_j-1}^{\,\gamma - T_j}$ is the 
	      lift of $s_{h_j-1}$ which passes through $\tilde x_{h_j-1}^\gamma$ we see 
	      that $\Ts_{h_j-1}^{\gamma - T_j}$ intersects $\Tt_{j,k}^{\,\gamma}$ at 
	      $\tilde w^\gamma$. 
	      \begin{center}
	        \includegraphics[height=4cm]{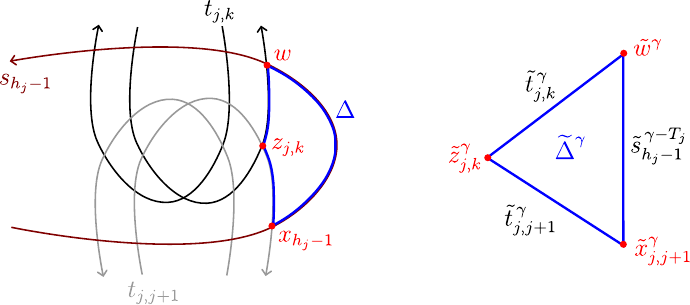}
	      \end{center}
	      We get that
	      \begin{equation*}
	        \renewcommand{\arraystretch}{2} 
	        \begin{array}{|c|c|} 
	          \hline
	          \kappa & \Tt_{j,k}^{\,\gamma} \cdot \Ts_{h_j-1}^{\,\gamma+\kappa} 
	          \\ \hline	\hline
	          - T_j \ , \ - 2T_j+T_k & 1 \\ \hline
	          - T_j + T_k \ , \ -2T_j & -1 \\ \hline
	        \end{array}
	      \end{equation*}
	      Therefore
	      \begin{equation*}
	        \inner{\phi_{j,k}, \omega_{h_j-1}} = 
	        \frac{\sqrt{-1}\left(\ov\rho_j + \ov\rho_j^2\rho_k 
	              - \ov\rho_j\rho_k - \ov\rho_j^2 \right)}
	             {(\ov\rho_j-1)(1-\rho_k)(\ov\rho_j-1)}
	        = \frac{\sqrt{-1} (1 - \ov\rho_j)(1-\rho_k)\,\ov\rho_j}
	               {(\ov\rho_j-1)(1-\rho_k)(\ov\rho_j-1)}
	        = \frac{\sqrt{-1}}{\rho_j-1}.
	      \end{equation*}
	      
	\item[(d')] \textbf{Computation of $\inner{\omega_{h_j-1}, \phi_{j,k} }$ 
	      when $h_{j-1} = h_j-1$:}
	      It can be seen that $\Tt_{j,k}^{\,\gamma}$ intersects 
	      $\Tt_{j-1,j}^{\,\gamma}$ at a pre-image of the point marked $w$ as shown below.
	      \begin{center}
	        \includegraphics[height=4cm]{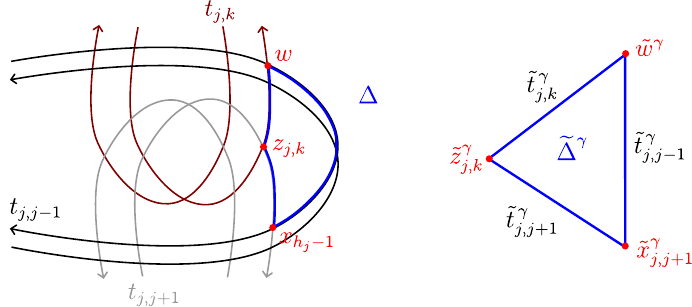}
	      \end{center}
	      Finding the other intersections is again a routine task:
	      \begin{center}
	        \includegraphics[width=.9\textwidth]{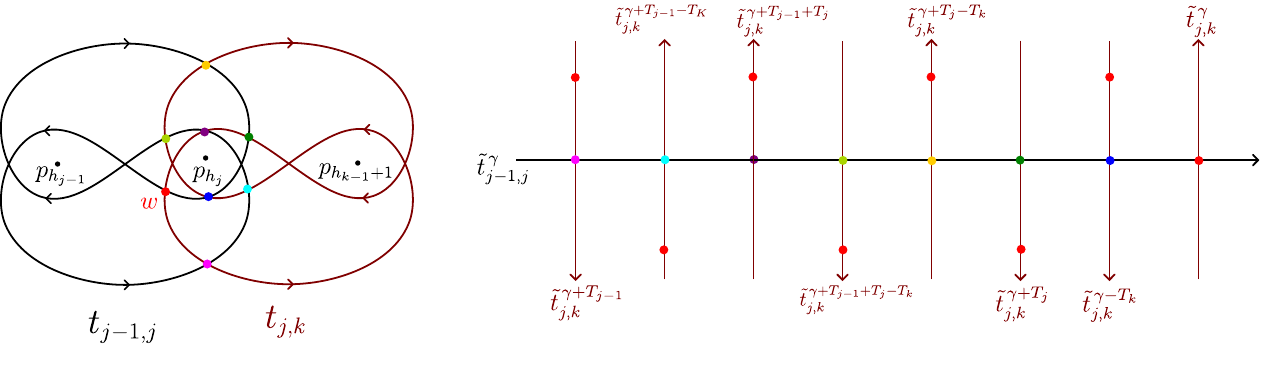}
	      \end{center}   
	      \begin{equation*} \renewcommand{\arraystretch}{2}
	        \begin{array}{|c|c|} \hline
	          \kappa & \Tt_{j-1,j}^{\,\gamma} \cdot \Tt_{j,k}^{\,\kappa} 
	          \\ \hline \hline
	          0 \ , \ T_{j-1}+T_j \ , \ T_j-T_k \ , \ T_{j-1}-T_k & 1 \\ \hline
	          T_{j-1} \ , \ T_j \ , \ - T_k \ , \ T_{j-1} + T_j  - T_k & -1 \\ \hline
	        \end{array}
	      \end{equation*}
	      Hence
	      \begin{align*}
	        \inner{\omega_{h_j-1}, \phi_{j,k}}
	          & = \frac{\sqrt{-1}(1 + \rho_{j-1}\rho_j 
	                    + \rho_{j-1}\ov\rho_k + \rho_j\ov\rho_k
	                    - \rho_{j-1} - \rho_j - \ov\rho_k 
	                    - \rho_{j-1}\rho_j\ov\rho_k)}
	                   {(\rho_{j-1}-1)(1-\ov\rho_j)(\rho_j-1)(1-\ov\rho_k)} \\
	          & = \frac{\sqrt{-1}(1-\rho_{j-1})(1-\rho_j)(1-\ov\rho_k)}
	                   {(\rho_{j-1}-1)(1-\ov\rho_j)(\rho_j-1)(1-\ov\rho_k)}
	            = \frac{\sqrt{-1}}{1-\ov\rho_j}.         
	      \end{align*}
	      
	\item \textbf{Computation of $\inner{\omega_{h_{k-1}+1}, \phi_{j,k} }$ 
	      when $h_{k-1}+1 < h_k$:}
	      The lifts $\Tt_{j,k}^{\,\gamma}$ and $\Ts_{h_{k-1}+1}^{\,\gamma}$ intersect
	      at a pre-image of the point marked $w$ as shown below.
	      \begin{center}
	        \includegraphics[height=4cm]{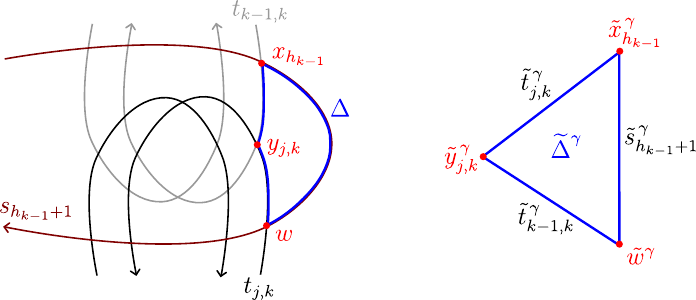}
	      \end{center}
	      From this information we can find all other intersections as usual.
	      \begin{center}
	        \includegraphics[width=.9\textwidth]{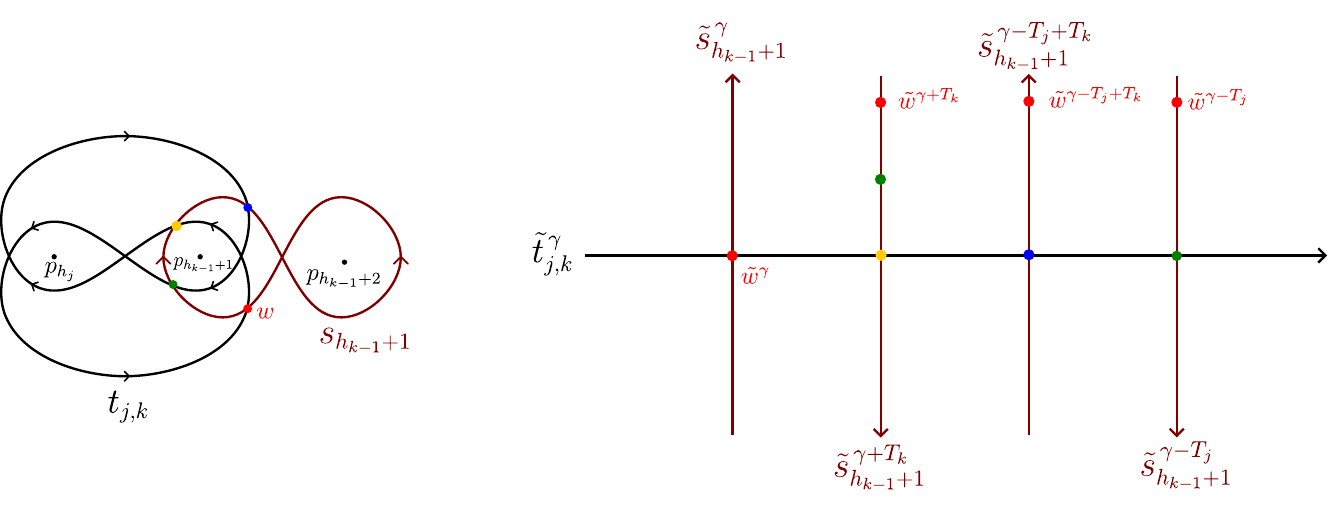}
	      \end{center}
	      It follows that
	      \begin{equation*}
	        \renewcommand{\arraystretch}{2} 
	        \begin{array}{|c|c|} 
	          \hline
	          \kappa & \Tt_{j,k}^{\,\gamma} \cdot \Ts_{h_{k-1}+1}^{\,\gamma+\kappa} \\ 
	          \hline \hline
	            0  \ , \ -T_j+T_k & 1 \\ \hline
	          -T_j \ , \ T_k & -1 \\ \hline
	        \end{array}
	      \end{equation*}
	      Hence, 
	      \begin{equation*}
	        \inner{\phi_{j,k}, \omega_{h_{k-1}+1}}  
	        = \frac{\sqrt{-1}(1 + \ov \rho_j\rho_k - \ov\rho_j - \rho_k)}
	               {(\ov\rho_j-1)(1 - \rho_k)(\ov\rho_k-1)}
	        = \frac{\sqrt{-1} (\ov\rho_j - 1)(\rho_k - 1)}
	               {(\ov\rho_j-1)(1 - \rho_k)(\ov\rho_k-1)} 
	        = \frac{\sqrt{-1}}{1-\ov\rho_k}.
	      \end{equation*}	
	      
	\item[(e')] \textbf{Computation of $\inner{\omega_{h_{k-1}+1}, \phi_{j,k} }$ 
	      when $h_{k-1}+1 = h_k$:}
	      In this case, it can be seen that $\Tt_{k,k+1}^{\,\gamma}$ intersects 
	      $\Tt_{j,k}^{\,\gamma}$ at a pre-image of the distinguished point $w$.  
	      \begin{center}
	        \includegraphics[height=4cm]{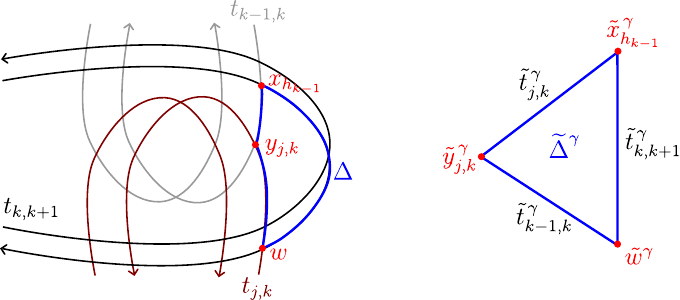}
	      \end{center}   
	      The other intersections are readily determined.
	      \begin{center}
	        \includegraphics[width=.9\textwidth]{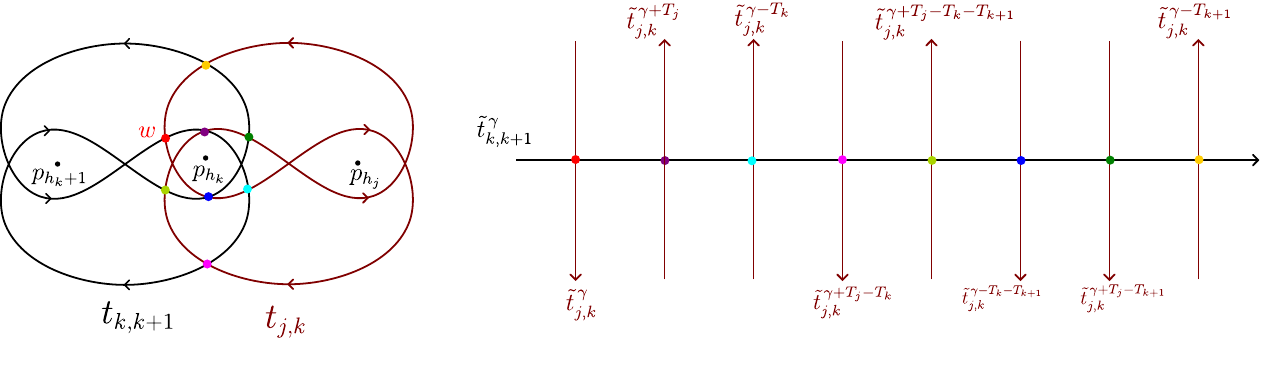}
	      \end{center}
	      \begin{equation*} \renewcommand{\arraystretch}{2}
	        \begin{array}{|c|c|} \hline
	          \kappa & \Tt_{k,k+1}^{\,\gamma} \cdot \Tt_{j,k}^{\,\gamma+\kappa} 
	          \\ \hline \hline
	          0 \ , \ T_j-T_k \ , \ T_j-T_{k+1} \ , \ -T_k-T_{k+1} & -1 \\ \hline
	          T_j \ , \ - T_k \ , \ - T_{k+1} \ , \ T_j - T_k  - T_{k+1} & 1 \\ \hline
	        \end{array}
	      \end{equation*}
	      Hence
	      \begin{align*}
	        \inner{\omega_{h_k},\phi_{j,k}} 
	          & = \frac{\sqrt{-1}(\rho_j + \ov\rho_k + \ov\rho_{k+1} 
	                    + \rho_j\ov\rho_k\ov\rho_{k+1} - \rho_j\ov\rho_k
	                    - \rho_j\ov\rho_{k+1} - \ov\rho_k\ov\rho_{k+1} -1)}
	                   {(\rho_k-1)(1-\ov\rho_{k+1})(\rho_j -1)(1-\ov\rho_k)} \\
	          & = \frac{\sqrt{-1}(\rho_j-1)(1-\ov\rho_k)(1-\ov\rho_{k+1})}
	                   {(\rho_k-1)(1-\ov\rho_{k+1})(\rho_j -1)(1-\ov\rho_k)} 
	            = \frac{\sqrt{-1}}{\rho_k-1}.
	      \end{align*}
  \end{enumerate}
  In all other cases the supports of $\phi_{j,k}$ and $\omega_i$ are disjoint.
\end{proof}

\begin{theorem} \label{thm:gens}
  The elements $\omega_1, \ldots, \omega_{n-1}$ span $H^1(X)_{\rho}$ for any 
  $\rho \in \mu'(d_1) \times \ldots \times \mu'(d_m)$. If 
  $\rho_1^{n_1} \cdots \rho_m^{n_m} \neq 1$ they are linearly 
  independent, otherwise they satisfy the following linear relation
  \begin{equation*} 
    \sum_{j=1}^m \ \sum_{i=h_{j-1}+1}^{h_j} 
    \left(1 - \rho_1^{n_1} \cdots \rho_{j-1}^{n_{j-1}}\,
                    \rho_j^{\,i - h_{j-1}}\right) \omega_i = 0.
  \end{equation*}
  Moreover for $1 \le j$ and $j+1<k \le m$
  \begin{equation*}
    \phi_{j,k} =  \omega_{h_j} + \omega_{h_j+1} + \ldots +\omega_{h_{k-1}}.
  \end{equation*}
\end{theorem}

\begin{proof}
  Let $M(P, \rho)$ denote the matrix
  \begin{equation*}
    M(P, \rho) := \Big(\inner{\omega_i, \omega_j}, 1 \le i,j \le n-1 \Big).
  \end{equation*}
  Since the only non-zero entries of $M$ occur on the diagonal and the
  adjacent off-diagonals, its determinant can be computed directly, giving
  \begin{equation*}
    \det M(P,\rho) = (\sqrt{-1})^{n-1} 
             \left( \frac{1 - \rho_1^{n_1} \cdots \rho_m^{n_m}}
                         {(1- \rho_1)^{n_1} \cdots (1 - \rho_m)^{n_m}} \right).
  \end{equation*}
  Thus when $\rho_1^{n_1} \cdots \rho_m^{n_m} \neq 1$ the
  matrix $M$ is non-singular and hence $\omega_1, \ldots, \omega_{n-1}$
  are linearly independent. Since in this case $\dim H^1(X)_\rho = n-1$, 
  this collection forms a basis. 
  
  When $\rho_1^{n_1} \cdots \rho_m^{n_m} = 1$, the matrix
  $M$ is singular. In fact in this case 
  \begin{equation*}
    \dim H^1(X)_\rho = n-2.
  \end{equation*}
  Let $M'$ be the $(n-2)\times (n-2)$ minor of $M$ obtained by deleting 
  its last row and last column. A similar reasoning as above
  shows that in this case $M'$ is invertible and therefore, 
  $\mathrm{rank}(M) = n-2$. Consequently the elements 
  $\omega_1, \ldots, \omega_{n-1}$ span $H^1(X)_\rho$. However,
  now they must satisfy a linear relation. To verify the relation
  we consider the element
  \begin{equation*}
    R = \sum_{j=1}^m \ \sum_{i=h_{j-1}+1}^{h_j} 
        \left(1 - \rho_1^{n_1} \cdots \rho_{j-1}^{n_{j-1}}\,
              \rho_j^{\,i - h_{j-1}}\right) \omega_i.
  \end{equation*}
  A direct computation shows that $\inner{R, \omega_i} = 0$ for all 
  $i=1, \ldots, n-1$, and therefore $R=0$.
  
  Finally, defining 
  \begin{equation*}
    \psi = \sum_{i=h_j}^{h_{k-1}} \omega_i
  \end{equation*}
  it is straight forward to verify $\inner{\phi_{j,k}-\psi, \omega_i} = 0$ 
  for all $i=1, \ldots, n-1$.
\end{proof}

\begin{corollary} \label{cor:span}
  The $\Gamma$-eigenspaces $H^1_c(\tD_i)_\rho$ and $H^1_c(\tE_{j,k})_\rho$ are 
  spanned by $\omega_i(\rho)$ and $\phi_{j,k}(\rho)$ respectively for any 
  $\rho \in \mu'(d_1) \times \ldots \times \mu'(d_m)$.
\end{corollary}

\begin{proof}
  For $\rho \in \mu'(d_1) \times \ldots \times \mu'(d_m)$, the eigenspaces
  $H^1_c(\tD_i)_\rho$ and $H^1_c(\tE_{j,k})_\rho$ are one dimensional by
  Proposition \ref{prop:subsurf}. Lemma \ref{lem:intomega} and Lemma \ref{lem:intphi} 
  show that the classes $\omega_i$ and $\phi_{j,k}$ are non-trivial, hence they
  must span the corresponding eigenspaces.
\end{proof}

%==================Third Section====================================================

\section{Monodromy Action of the Mixed Braid Group}

Recall from \eqref{eq:hom} the homomorphism $\lift: \B_{n,\Lambda} \to \mcg(X)^\Gamma$ 
given by $\gamma \mapsto \wt\gamma$ induces the representation of $\B_{n,\Lambda}$ on
$H^1(X)_{\rho}$ by $\gamma \cdot \xi = (\wt\gamma^{-1})^*\xi$.

In this section we shall describe this representation for 
$\rho \in \mu'(d_1) \times \cdots \times \mu'(d_m)$. We may restrict our focus to 
this case because if $\rho_j = 1$ for some index $j$ then 
\begin{equation*}
  H^1(X)_{\rho} \cong H^1(X/(T_j))_{\rho'}  \tx{where} 
  \rho' = (\rho_1, \ldots, \widehat{\rho_j}, \ldots, \rho_m)
\end{equation*}
as $\cc[\Gamma/(T_j)]$-modules. Our arguments allow us to compute the representation
when $\rho_j\rho_k \neq 1$ for any $j,k \in \{1,\ldots, m\}$ since we need the self 
intersections $\inner{\omega_i, \omega_i}$ and $\inner{ \phi_{j,k}, \phi_{j,k} }$
to be non-trivial.

\begin{proposition} \label{prop:kernel}
  Let $\tau = (\sigma_1 \cdots \sigma_{n-1})^n$, then 
  $\wt \tau = T_1^{-n_1}\cdots T_m^{-n_m}$ in $\mcg(X)$ and
  \begin{equation*}
    \ker\Big( \lift: \B_{n,\Lambda} \to \mcg(X)^\Gamma \Big) = \left(\, \tau^f \,\right)
  \end{equation*}
  where $f = \ord\big( n_1T_1 + \ldots + n_mT_m \big)$.
\end{proposition} 

\begin{proof}
  There is the capping homomorphism 
  \begin{equation*}
    \Cap : \mcg_c(\cc, P) \to \mcg(\cc, P)
  \end{equation*}
  whose kernel is the cyclic subgroup generated by $\tau$ (see 
  \cite[Proposition 3.19]{FM}). There is also a homomorphism 
  $\mcg(X)^\Gamma \to \mcg(\cc, P)$ since the cover $\pi: X \to \pp^1$ has the 
  Birman-Hilden property (\cite[Corollary 12]{HM}). The composition
  \begin{equation*}
    \B_{n,\Lambda} \xrightarrow{\lift} \mcg(X)^\Gamma \to \mcg(\cc, P)
  \end{equation*}
  is just $\Cap$ restricted to $\B_{n,\Lambda}$. Hence $\ker (\lift) \subset 
  \ker(\Cap) = (\tau)$.
  
  Since $\tau$ is in the kernel of $\Cap$ the lift $\wt\tau$ is homotopic
  to a deck transformation. Moreover, $\wt\tau$ fixes all the pre-images of $\infty$.
  Thus $\wt \tau  \simeq S^l$ for some $l$, where 
  $S = n_1T_1 + \ldots + n_mT_m$ is a generator of the stabiliser
  subgroup (of the deck group) for any pre-image of infinity.

  Choose $a > 0$ such that $|p_i| < a$ for all $i = 1, \ldots, n$. The mapping class 
  $\tau$ is represented by the Dehn twist supported on the annulus 
  $a \leq |z| \leq a+1$:
  \begin{equation*}
    \tau(z) = 
    \begin{cases}
      \exp \big(2\pi \sqrt{-1}(|z|-a)\big)\, z, & a \leq |z| \leq a+1, \\
      z, & \text{otherwise}.           
    \end{cases}
  \end{equation*}
  Consider the homotopy $H_t$ from $H_0 = \tau$ to $H_1 = \id$ 
  defined by:
  \begin{equation*}
    H_t(z) = 
    \begin{cases}
      z & |z|< a, \\
      \exp \big(2\pi \sqrt{-1} (1-t)(|z|-a)\big) \, z, & a \leq |z| \leq a+1, \\
      \exp \big(2\pi \sqrt{-1} (1-t) \big) \, z, & |z| > a+1.           
    \end{cases}
  \end{equation*}
  Note that this homotopy is not compactly supported. 
  Let $\wt H_t$ be the unique lift of this homotopy such that 
  $\wt H_0 = \wt\tau$. Then we must have $\wt H_1 = S^l$.
  
  On a small disk $B$ centred at $\infty$ the homotopy $H_t$ is given by
  $H_t(w) = \exp(2\pi \sqrt{-1} (t-1))\, w$ in the local coordinate $w = 1/z$. Let 
  $\wt\infty_0$ be a pre-image of $\infty$ in $X$ and $\wt B_0$ be the component of 
  $\pi^{-1} B$ which contains $\wt\infty_0$. In suitable coordinates, the map 
  $\pi: \wt B_0 \to B$ is a branched covering of the form $\pi(w) = w^f$.  
  It follows that the restriction of the lift $\wt H_t$ to $\wt B_0$ is 
  \begin{equation*}
    \wt H_t(w) = \zeta_f^k \, \exp \big(2\pi \sqrt{-1}(t-1)/f) \, w 
    \tx{for some} k \in \{ 0, \ldots, f-1 \}.
  \end{equation*}
  Since $\wt H_0 = \wt\tau$ is identity on $\wt B_0$ we see that $k=1$.
  Evaluating at $t=1$, we obtain 
  \begin{equation*}
    S^l(w) = \wt H_1(w) = \zeta_f \, w 
    \tx{on} \wt B_0.
  \end{equation*}
  
  Now let us find out how $S$ acts on $\wt B_0$. Note that $X$ is the completion of
  \begin{equation*}
    \{ (x,y_1, \ldots, y_m) \in \cc^{m+1} \mid y_j^{d_j} = 
                   (x-p_{h_{j-1}+1}) \cdots (x - p_{h_j}) \}.
  \end{equation*}
  The function $x$ has a pole of order $f$ at $\wt \infty_0$ and $y_i$ have poles 
  of order $fn_i/d_i$. It can be seen that $f, fn_1/d_1, \ldots, fn_m/d_m$ are 
  co-prime, hence there are integers  $\mu, \nu_1, \ldots, \nu_m$ such that 
  \begin{equation*}
    \mu f + \nu_1\left(\frac{fn_1}{d_1} \right) + \ldots 
          + \nu_m\left(\frac{fn_m}{d_m} \right) = 1.
  \end{equation*} 
  Thus $w = x^{-\mu}y_1^{-\nu_1} \cdots y_m^{-\nu_m}$ is a local coordinate at
  $\wt \infty_0$. In this coordinate 
  \begin{equation*}
    S(w) = \zeta_{d_1}^{-n_1\nu_1} \cdots \zeta_{d_m}^{-n_m\nu_m} \, w
         = \zeta_f^{-1} w.
  \end{equation*}
  Thus it follows that $\wt \tau \simeq S^{-1}$.
\end{proof}

Let $V$ be a finite-dimensional complex vector space. A linear transformation 
$s \in \text{GL}(V)$ is called a \textbf{complex reflection} if it has finite 
order and its fixed-point set 
\[
V^s = \{v \in V \mid s(v) = v\}
\]
is a hyperplane (a subspace of codimension 1).

\begin{theorem} \label{thm:main}
  Let $\rho \in \mu'(d_1) \times \cdots, \mu'(d_m)$ such that
   $\rho_j \neq \ov\rho_k$ for any $j,k \in \{1, \ldots,m\} $. 
  In that case the action of the generators of the mixed braid group 
  $\B_{n,\Lambda}$ on $\xi \in H^1(X)_{\rho}$ is given by complex 
  reflections as follows
  \begin{align*}
    \sigma_i(\xi)  
      & = \xi - (\rho_j+1) \frac{\inner{\xi, \omega_i}}
                         {\inner{\omega_i, \omega_i}} \, \omega_i 
      & h_{j-1} < i < h_j, \\[10pt]
    A_{j,k}(\xi) 
      & = \xi + \big( \rho_j\rho_k-1 \big) \frac{\inner{\xi, \phi_{j,k}}}
              {\inner{\phi_{j,k}, \phi_{j,k}}} \phi_{j,k}
      & 1 \le j < k \le m,
  \end{align*}
  where we define $\phi_{j,j+1} = \omega_{h_j}$.
\end{theorem}

\begin{proof}
  First note that $\sigma_i$ is supported on the disk $\D_i$.
  The eigenspace $H^1_c(\tD_i)_\rho$ is one-dimensional and spanned by $\omega_i$.
  Suppose $h_{j-1} < i < h_j$, then it follows from \cite[Theorem 4.4]{Mc} that 
  \begin{equation*}
    \sigma_i \cdot \omega_i = -\rho_j \omega_i.
  \end{equation*}
  We have a $\sigma_i$ equivariant exact sequence
  \begin{equation*}
    0 \to H^1_c(\tD_i)_{\rho} \to H^1(X)_{\rho} \to H^1(X - \tD_i)_{\rho} \to 0.
  \end{equation*}
  Since $\sigma_i$ acts trivially on $H^1(X - \tD_i)_{\rho}$ the eigenvalues of 
  $\sigma_i$ on $H^1(X)_\rho$ are $-\rho_j$ and $1$. By our assumption
  $\rho_j \neq -1$ and hence 
  $\inner{\omega_i, \omega_i} = 2 \im \rho_j \ne 0$. 
  Thus we have an orthogonal decomposition
  \begin{equation*}
    H^1(X)_{\rho} = H^1_c(\tD_i)_{\rho} \oplus H^1_c(\tD_i)_{\rho}^\perp
  \end{equation*}
  with respect to the inner product $\inner{\cdot, \cdot}$. Thus 
  $H^1_c(\tD_i)_{\rho}^\perp$ must be the $1$-eigenspace of $\sigma_i$.
  
  Now turning to $A_{j,k}$ we note that it is supported on the disk $\E_{j,k}$.
  As in proof of Proposition \ref{prop:subsurf} we can identify $\E_{j,k}$
  with the open disk of radius $2$ in $\cc$, in which case we have an isomorphism
  \begin{align*}
    \tE_{j,k}^{\gamma} & \cong \{ (x,y) \in \cc^2 \mid y^{d_k} = x^{d_j} - 1, \quad
                       |x^{d_j}| < 2 \}.
  \end{align*}
  Under these identifications the branched covering 
  $\pi: \tE_{j,k}^{\gamma} \to \E_{j,k}$ is given by 
  \begin{equation*}
    \pi(x,y) = x^{d_j}
  \end{equation*}  
  and the deck transformations are $T_j(x,y) = (\zeta_{d_j} x, y)$ and 
  $T_k(x,y) = (x, \zeta_{d_k} y)$.
  The mapping class $A_{j,k}$ on $\E_{j,k}$ is represented by the homeomorphism 
  (see Figure \ref{fig:generators})
  \begin{equation*}
    A_{j,k}(z) =
    \begin{cases}
      \exp\big(2\pi \sqrt{-1} |z|\big)\, z & |z| \le 1, \\
      \exp\big(2\pi \sqrt{-1}(2- |z|)\big) \, z & 1 \le |z| \le 2. 
    \end{cases}
  \end{equation*}
  Let $B := \{x \in \cc \mid |x|< 2^{1/d_j} \}$, then $\pi$ is the composition
  \begin{equation*}
    \tE_{j,k}^{\gamma} \xrightarrow{\pi'} B \xrightarrow{\pi''} \E_{j,k}
  \end{equation*}
  where $\pi'(x,y) = x$ and $\pi''(x) = x^{d_j}$. The lift of $A_{j,k}$ to 
  $B$ under $\pi''$, (which is identity on the boundary), is 
  \begin{equation*}
    \begin{cases}
      \exp\big(2\pi \sqrt{-1} |z|^{d_j}/d_j\big)\, z & |z| \le 1, \\
      \exp\big(2\pi \sqrt{-1}(2- |z|^{d_j})/d_j \big) \, z & 1 \le |z| \le 2. 
    \end{cases}
  \end{equation*} 
  This is clearly homotopic on $B - (\{0\} \cup \mu(d_j))$ to the rotation
  by $\zeta_{d_j}$. Hence by \cite[Proposition 4.2]{Mc} we have
  \begin{equation*}
    \wt A_{j,k} \simeq T_jT_k \tx{on} \tE_{j,k}^\gamma.
  \end{equation*}
  Thus we have 
  \begin{equation*}
    A_{j,k} \cdot \phi_{j,k} = \rho_j\rho_k \phi_{j,k}.
  \end{equation*}
  Again since $\rho_j \ne \ov \rho_k$ the self-intersection 
  $\inner{ \phi_{j,k}, \phi_{j,k}} = 2 \im (\rho_j + \rho_k - \rho_j\rho_k) \ne 0$.
  The rest of the argument is similar to the previous case.
\end{proof}

\begin{remark} \label{rem:matrices}
  In the basis $\omega_1, \ldots, \omega_{n-1}$ the matrices of $\theta_\rho$ 
  for the generators of $\B_{n, \Lambda}$  are as follows:\\
  \textbf{For the generators} $\sigma_i$
  \begin{equation*}
    \theta_\rho(\sigma_i) = 
    \begin{cases} 
      \Mat{\mat{-\rho_1 & 1 \\ 0 & 1 }& 0 \\ 0 & \id_{n-3}}
        & \text{ if } i=1, \ h_1 > 1 \\[2em]
      \Mat{\id_{n-3}& 0 \\ 0& \mat{ 1 & 0 \\ \rho_m & -\rho_m}} 
        & \text{ if } i=n-1, \ h_{m-1} < n-2 \\[2em]
      \Mat{\id_{i-2} & 0 & 0 \\ 
           0 &\mat{ 1 & 0 & 0 \\ \rho_j & -\rho_j & 1 \\ 0 & 0 & 1 } & 0\\ 
           0 & 0 & \id_{n-i-2}} 
        & \text{ for } 1 < i < n-1.
    \end{cases}
  \end{equation*}  
  \textbf{For the generators} $A_{j,k}$
  \begin{equation*}
    \theta_\rho(A_{j,k}) = \id_{n-1} \ + \ S_{j,k}R_{j,k} 
    \tx{for} 1 \le j < k \le m \\[1em]
  \end{equation*}
  where the column vector $S_{j,k}$ and the row vector $R_{j,k}$ are
  \begin{align*}
    S_{j,k} & = \Big(0 , \ldots , 0 , \underbrace{1}_{h_j} , 1 , 
                         \ldots , 1 , \underbrace{1}_{h_{k-1}} , 0 
                         , \ldots , 0 \Big)^T 
                 \tx{and}  \\[1em]
    %-------------------------------------------------------------------           
    R_{j,k} & = 
    \begin{cases}
      \Big(0, \ldots, 0, \,
           \rho_j(1- \rho_k), \, \underbrace{\rho_j\rho_k-1}_{h_j}, \,
           1-\rho_j, 0, \ldots, 0 \Big) 
        & \text{ if } k = j+1, \\
      \Big(0, \ldots, 0, \, \rho_j(1- \rho_k) , \, \underbrace{\rho_k-1}_{h_j}, \, 
           0, \ldots, 0, \, \, \underbrace{(\rho_j-1)\rho_k}_{h_{k-1}}, \, 
           1-\rho_j, \, 0 , \ldots , 0 \Big) 
        & \text{ if } k > j+1.
    \end{cases}
  \end{align*}
  \textit{Note:} If $h_1=1$, the first entry of $R_{1,k}$ is omitted. 
  If $h_{m-1}=n-1$, the last entry of $R_{j,m}$ is omitted. 
\end{remark}

The proof of the following result is similar to that of Proposition 5.1 of \cite{Mc}.

\begin{proposition} \label{prop:irred}
  The representation $\theta_\rho$ is irreducible for any
  $\rho \in \mu'(d_1) \times \cdots \times \mu'(d_m)$ such that $\rho_j\rho_k \ne 1$ 
  for any $1 \le j,k \le m$ .    
\end{proposition}

\begin{proof}
  Let $N(\neq 0)$ be a proper invariant subspace of $H^1(X)_\rho$. Let $v$ be a 
  non zero element of $N$. We can find $\omega_i$ such that the intersection 
  $\inner{v,\omega_i} \neq 0$. Then, if $i \neq h_j$ for any $j=1, \ldots, m$
  \begin{equation*}
    v- \sigma_i(v)= 
    \frac{\inner{v,\omega_i}}{\inner{\omega_i, \omega_i}} \omega_i
    \in N \quad \Rightarrow \quad \omega_i\in N.
  \end{equation*}
  Similarly if $i=h_j$ for some $j$ we must have $\omega_{h_j} \in N$.
  Again if $i< n-1$ we have $\inner{\omega_i, \omega_{i+1}} \neq 0$, so
  $\omega_{i+1}\in N$. If $i>1$, $\inner{\omega_i, \omega_{i-1}} \ne 0$ so
  $\omega_{i-1}\in N$. If we continue this process we see that 
  $N = H^1(X)_\rho$. 
\end{proof}

\subsection*{The Multivariate Burau representation} 
The multivariate Burau representation
\begin{equation*}
  \wt \beta: \B_{n,\Lambda} \to 
  \mathrm{GL}_n \left(\zz[t_1^{\pm 1}, \ldots, t_m^{\pm 1}] \right)
\end{equation*}
arises from the action of $\B_{n,\Lambda}$ on the homology, (with coefficients in 
$\zz[t_1^{\pm 1}, \ldots, t_m^{\pm 1}]$), of the infinite Abelian cover of the 
punctured plane $\mathbb{C} \setminus \{p_1, \ldots, p_n\}$, obtained from the
homomorphism
\begin{equation*}
  \Pi_1(\cc - \{p_1,\ldots, p_n\}) \to \zz^m, \qquad 
  c_i \mapsto (0,\ldots, 0, \underbrace{1}_j, 0, \ldots, 0\}
  \text{ for } h_{j-1}<i \le h_j.
\end{equation*}
This is a type of Magnus representation (see \cite[Chapter 3]{Bir}). 
Explicitly, the action of the generators is given by the matrices
\begin{equation*}
  \wt\beta(\sigma_i) = 
  \Mat{ \id_{i-1} & 0 & 0 \\
        0 & \mat{1 - t_j & 1 \\ t_j & 0} & 0 \\
        0 & 0 & \id_{n-i-1} }
  \qquad h_{j-1} < i < h_j
\end{equation*}
and for $1 \le j < k \le m$
\begin{equation*}
  \wt{\beta}(A_{j,k}) =  \id_n \ + \ \wt R_{j,k} \, \wt S_{j,k}             
\end{equation*}
where the column vector $\wt R_{j,k}$ and row vector $\wt S_{j,k}$ are defined as:
\begin{align*}
  \wt R_{j,k} & = (1-t_k) e_{h_j} + (t_j-1) e_{h_{k-1}+1}  \\        
  \wt S_{j,k} & = -t_j e_{h_j}^T \ + \ \sum_{r=j+1}^{k-1} 
                                       \sum_{i=h_{r-1}+1}^{h_r} (1-t_r)e_i^T
                   \ + \ e_{h_{k-1}+1}^T.
\end{align*}
These are calculated using formula 3-19 of \cite{Bir}.
(Here as usual $e_1, \ldots, e_n$ are the columns of the identity matrix.
We have transposed the matrices from \cite{Bir} to account for the difference in 
convention; while \cite{Bir} uses a right action with row vectors, we 
have used the more standard left action with column vectors.)

The representation $\widetilde{\beta}$ is reducible. Specifically, the submodule
spanned by 
\begin{equation*}
  \sum_{j=1}^m (1-t_{j(i)}) e_i, \tx{where}
  j(i) \in \{1, \ldots, m\} \text{ such that } h_{j(i)-1}< i \le h_{j(i)}
\end{equation*}
is invariant and carries the trivial representation. By passing to the quotient, 
we obtain the reduced multivariate Burau representation:
\begin{equation*}
  \beta : \B_{n,\Lambda} \to 
  \mathrm{GL}_{n-1} \left(\zz[t_1^{\pm 1}, \ldots, t_m^{\pm 1}] \right).
\end{equation*}
For any choice of roots of unity $\rho = (\rho_1, \ldots, \rho_m)$, let 
$\mathrm{ev}_\rho: \mathbb{Z}[t_1^{\pm 1}, \ldots, t_m^{\pm 1}] \to \mathbb{C}$ be the 
evaluation map $t_i \mapsto \rho_i$. The composition of $\beta$ with $\mathrm{ev}_\rho$ yields 
a complex representation
\begin{equation*}
  \beta_\rho : \B_{n,\Lambda} \to \mathrm{GL}_{n-1}(\cc)
\end{equation*}
referred to as the reduced multivariate Burau representation evaluated at $\rho$.

Recall the dual representation $\beta_\rho^*$ is defined as
\begin{equation*}
  \beta_\rho^*(b) = \beta_\rho(b^{-1})^T, \qquad b \in \B_{n,\Lambda}.
\end{equation*}

\begin{proposition} \label{prop:burau}
  The representation $\theta_\rho$ is isomorphic to the dual $\beta_{\ov\rho}^*$ of the 
  reduced multivariate Burau representation evaluated at $\ov\rho$, when 
  $\rho_j\rho_k \neq 1$ for any $j,k \in \{1,\ldots, m\}$ and 
  $\prod_j \rho_j^{n_j} \neq 1$.
\end{proposition}

\begin{proof}
  Let $\wt \beta_\rho = \mathrm{ev}_\rho \circ \wt \beta$, then we shall show that
  \begin{equation*}
    \wt \beta_{\ov\rho}^* \cong \theta_\rho\oplus \mathbf{1}
    \tx{where} \mathbf{1} \text{ is the trivial representation.}
  \end{equation*}
  Let $v_i = -\ov \rho_{j(i)} e_i + e_{i+1}$, then
  \begin{itemize} \itemsep .5em
  \item for $i\neq h_j$ the matrix $\beta_{\ov \rho}(\sigma_i)^T$ is a complex reflection with
        $\beta_{\ov \rho}(\sigma_i)^T v_i = \ov \rho_{j(i)}v_i$ and
  \item $\beta_{\ov \rho}(A_{j,j+1})^T$ is a complex reflection with
        $\beta_{\ov \rho}(A_{j,j+1})^T v_{h_j} = \ov \rho_j\ov \rho_{j+1}v_{h_j}$.      
  \end{itemize}
  Moreover, let $w = \sum_{i=1}^n (1- \rho_{j(i)}) e_i$ then
  \begin{equation*}
    \wt\beta_{\ov\rho}(\sigma_i)^Tw = 
    \wt\beta_{\ov\rho}(A_{j,k})^Tw = w, \qquad 
    \text{for all } 1 \le j<k \le n \text{ and } i \neq h_j.
  \end{equation*}
  Thus we work with the basis $v_1, \ldots, v_{n-1}, w$. 
  Let $B$ be the change of basis matrix
  \begin{equation*}
    B = 
    \mat{
      -\ov\rho_{j(1)} & 0 & 0 & & \cdots & 0 & 1 - \ov\rho_{j(1)} \\
       1 & -\ov\rho_{j(2)} & 0 & & \cdots & 0 & 1- \ov\rho_{j(2)} \\
       0 & 1 & \ddots &  \\
       0 & 0 & \ddots \\
      \vdots & \vdots & &  & & \vdots & \vdots\\      
      0 & 0 & & & \cdots & -\ov\rho_{j(n-1)} & 1 - \ov\rho_{j(n-1)} \\  
      0 & 0 & & & \cdots & 1 & 1 - \ov\rho_{j(n)} 
    }
    \quad \in \quad \mathrm{GL}_n(\cc).
  \end{equation*}  
  Then, a direct computation shows
  \begin{equation*}
    B^{-1} \left(\wt\beta_{\ov \rho}(\sigma_i)^T \right)^{-1} B = 
    \Mat{ \theta_\rho(\sigma_i) & 0 \\ 0 & 1} 
    \tx{and}
    B^{-1} \left(\wt\beta_{\ov \rho}(A_{j,k})^T \right)^{-1} B = 
    \Mat{ \theta_\rho(A_{j,k}) & 0 \\ 0 & 1}. 
  \end{equation*}
  Since $\wt \beta_\rho \cong \beta_\rho \oplus \mathbf{1}$ the result follows.
\end{proof}

\subsection*{Image of the representation}
The image $\theta_\rho(\B_{n,\Lambda})$ is contained in 
$\mathrm{GL}_{n-1}\big(\zz[\rho_1, \ldots, \rho_m]\big)$. It will be a very 
interesting problem to classify the partitions $\Lambda$ and the roots of unity
$\rho$ for which the image is an arithmetic subgroup.

If $P = (n_1, \ldots, n_m)$ and $n_i > 1$ for some $i$, we have 
$\B_{n_i} \subset \B_{n,\Lambda}$. Let 
\begin{equation*}
  \rho_i = \zeta_{d_i}^{-k_i}, \qquad r_i = \lceil n_i(k_i/d_i) - 1 \rceil
  \tx{and} s_i = \lceil n(1 - k_i/d_i) - 1\rceil
\end{equation*}
then $\theta_\rho$ induces a homomorphism
\begin{equation*}
  \B_{n_i} \to \U(r_i, s_i) \subset U(r,s)
\end{equation*}
obtained from the action of $\B_{n_i}$ on the subspace spanned by 
$\omega_{h_{i-1}+1}, \ldots, \omega_{h_i-1}$. This is just the representation
given in \cite[Section 2, (2.4)]{Mc}. Hence it follows from 
\cite[Theorem 8.1]{Mc} that if $n_i > 2$ for some $i$ and
\begin{equation*}
  (n_i, d_i) \not \in \{(3,3), (3,4), (3,6), (3,10), (4,4), (4,6)
                        (5,6), (6,6) \}
\end{equation*}
$\theta_{\rho}(\B_{n,\Lambda})$ is infinite. This is the vast majority of the cases.

As an example not covered by the previous cases is  
$ \{(x,y_1,y_2) \in \cc^3 \mid y_1^3 = x^2-1, y_2^5 = x^2+1\}$, which has genus
$15$. Taking $\rho = (\zeta_3, \zeta_5)$ we see that the image of $\theta_\rho$ in 
$U(1,2)$ is generated by
\begin{equation*}
  \Mat{-\zeta_3 & 1 & 0 \\ 0 & 1 & 0 \\ 0 & 0 & 1} , \quad
  \Mat{1 & 0 & 0 \\ \zeta_3-\zeta_{15}^8 & \zeta_{15}^8 & 1-\zeta_3 
                 \\ 0 & 0 & 1} , \quad
  \Mat{1 & 0 & 0 \\ 0 & 1 & 0 \\ 0 & \zeta_5 & -\zeta_5}.
\end{equation*}
We suspect this to be an infinite group, in which case the question of discreteness
arises.

%==================Notation=========================================================

\section*{List of symbols}

\begin{description} \itemsep 5pt
\item[$\zeta_d$] Primitive $d$-th root of unity $e^{2\pi i/d}$.
\item[$\mu(d)$] $d$-th roots of unity.
\item[$\mu'(d)$] $d$-th roots of unity except 1.
\item[$\mcg(X)$] Mapping class group of a surface $X$.
\item[$\mcg_c(X)$] Compactly supported mapping class group of a (non-compact)
      surface $X$.
\item[$\B_n$] Braid group of $n$-strands.
\item[$\B_{n,\Lambda}$] Mixed braid group for a partition $\Lambda$ of 
      $\{1, \ldots, n\}$. 
\item[$g(X)$] genus of a Riemann surface $X$ of finite type.  
\end{description}

\subsection*{Acknowledgement} We are grateful to Pranav Haridas for many fruitful 
discussions. The authors are also grateful to Sandro Manfredini for an important 
clarification regarding the presentation of the mixed braid groups.

%==================References=======================================================

\end{document}